\documentclass[11pt]{amsart}

\usepackage[T1]{fontenc} 
\usepackage{lmodern}
\usepackage{microtype}
\usepackage{amsmath,amssymb,mathtools}
\usepackage{xcolor}
\usepackage{enumitem}
\usepackage[colorlinks=true,linkcolor=blue,citecolor=blue,urlcolor=blue]{hyperref}

\allowdisplaybreaks
\newcommand{\C}{\mathbb C}
\newcommand{\R}{\mathbb R}
\newcommand{\Z}{\mathbb Z}
\newcommand{\cO}{\mathcal O}
\newcommand{\cH}{\mathcal H}
\newcommand{\cP}{\mathcal P}
\newcommand{\cS}{\mathcal S}
\newcommand{\Ch}{\operatorname{Ch}}
\newcommand{\Irr}{\operatorname{Irr}}
\newcommand{\Ind}{\operatorname{Ind}}
\newcommand{\Hom}{\operatorname{Hom}}
\newcommand{\Lie}{\operatorname{Lie}}
\newcommand{\Ad}{\operatorname{Ad}}
\newcommand{\Stab}{\operatorname{Stab}}
\newcommand{\soc}{\operatorname{soc}}
\newcommand{\triv}{\mathbf 1}
\newcommand{\IM}{\operatorname{IM}}
\newcommand{\LLC}{\operatorname{LLC}}
\newcommand{\wt}{\operatorname{wt}}

\theoremstyle{plain}
\newtheorem{theorem}{Theorem}[section]
\newtheorem{proposition}[theorem]{Proposition}
\newtheorem{lemma}[theorem]{Lemma}
\newtheorem{corollary}[theorem]{Corollary}

\theoremstyle{definition}
\newtheorem{definition}[theorem]{Definition}

\theoremstyle{remark}
\newtheorem{remark}[theorem]{Remark}

\title[Enhanced Langlands parameters]{Enhanced Langlands Parameters of Regular Principal Series}
\author{Haoming Li}
\address{Cuiying Honors College, Lanzhou University,
No.~222 Tianshui South Road, Lanzhou 730000, China}
\email{lihaoming2023@lzu.edu.cn}

\author{Xiaohang Yang}
\address{Beijing International Center for Mathematical Research,
Peking University, No.~5 Yiheyuan Road, Haidian District,
Beijing 100871, China}
\email{yxh2219@stu.pku.edu.cn}

\author{Jizhi Zhang}
\address{School of Mathematics and Statistics, Lanzhou University,
No.~222 Tianshui South Road, Lanzhou 730000, China}
\email{zhangjzh2024@lzu.edu.cn}

\author{Yelei Zhao}
\address{School of Mathematical Sciences, Zhejiang University,
No.~866 Yuhangtang Road, Hangzhou 310058, China}
\email{3230104016@zju.edu.cn}
\date{}

\begin{document}

\begin{abstract}
We describe the enhanced Langlands parameters of the irreducible
constituents of a regular principal series of a quasi-split reductive group
over a non-archimedean local field of characteristic zero.  Starting from
the chamber parametrization by Rodier and Luo, we determine the monodromy
of each constituent and show that its enhancement is trivial.  In
particular, we obtain the corresponding parameter of the unique irreducible
subrepresentation.  
\end{abstract}

\maketitle

\setcounter{tocdepth}{1}
\tableofcontents
\clearpage

\section{Introduction}

A motivating question is when the Langlands parameter of the unique
irreducible subrepresentation of a regular principal series has zero
monodromy. For a quasi-split group, it is natural to seek a criterion
formulated directly in terms of the co-rank-one reducibility data over
the ground field.

Bella\"iche and Chenevier considered the vanishing question for the local monodromy in
\cite[Section~6.6]{BC09}, as a local condition for their construction of Galois representations. They passed to a splitting field and pulled
back the inducing character by the norm map.  They required the
resulting character to be regular and imposed a sign condition on
the absolute reducibility coroots.
They also discussed the possibility of an intrinsic criterion for
the original quasi-split group
\cite[Remark~6.6.6(4)]{BC09}.

The constituents of a regular principal series
have a parametrization in terms of chambers cut out by the co-rank-one
reducibility hyperplanes.  This parametrization is due to Rodier for split
groups \cite{Rod81}.  Luo extended it to regular generalized principal
series of reductive groups and proved the linear independence
of the reducibility coroots in that setting
\cite[Theorems~3.4 and~4.1]{Luo21}.  We use the chamber parametrization by
Rodier and Luo in this general form.  In particular, the constituent
classification, the Jacquet modules, and multiplicity one are known.

On the geometric side, Kazhdan and Lusztig established the
Deligne--Langlands classification for affine Hecke algebras \cite{KL87}.
For split groups, Roche constructed principal-series types and described
their Hecke algebras \cite[Theorems~6.3 and~7.7]{Roc98}.  Reeder used
this framework to give a geometric classification of ramified principal
series for split groups with connected centre \cite[Theorem~1]{Ree02}.
These results of Roche and Reeder have residual-characteristic
restrictions.  For the arbitrary quasi-split groups considered here, we
use Solleveld's canonical Whittaker-normalized principal-series
correspondence \cite[Theorems~5.4, 6.9, and~7.1]{Sol23}, together with
the geometric Hecke algebras of Aubert, Moussaoui, and Solleveld
\cite{AMS16,AMS17}.

The relation with chambers also appears in the theory of calibrated
Hecke modules.  Ram described regular affine principal-series
constituents by local regions and constructed their weight spaces
\cite[Theorems~2.12(a), 3.2, 3.5, and~3.6]{Ram03}; Kriloff and Ram
gave the graded version with arbitrary parameters
\cite[Theorems~4.5 and~4.7]{KR02}.  In particular, Ram's regular
$A_2$ example already associates subsets of roots with sums of root
vectors and trivial enhancements \cite[Table~4.1]{RamRankTwo03}.

Theorem~\ref{thm:main} determines the enhanced Langlands parameter
of each irreducible constituent directly from its chamber label.
For the unique irreducible subrepresentation, it gives a
zero-monodromy criterion using only the reducibility walls over $F$.
Proposition~\ref{prop:bc-comparison} shows that the conditions of
Bella\"iche and Chenevier imply this criterion.

The proof uses Solleveld's correspondence and its reductions for
ramified characters and nonsplit groups. We compute selected weights
in geometric standard modules and analyze the component-group action,
including when the centralizer is disconnected. These calculations,
together with the geometric multiplicity formula, identify the
nilpotent orbit and enhancement attached to each chamber.

\subsection{Notation and the main theorem}

Let $F$ be a non-archimedean local field of characteristic zero, with residue
cardinality $q_F$, and fix a uniformizer $\varpi_F$ normalized by
$|\varpi_F|=q_F^{-1}$.  Let $G$ be a connected quasi-split reductive group over
$F$, let $B=TU$ be an $F$-Borel subgroup, and let $S$ be the maximal $F$-split
subtorus of $T$.  We write
\[
  W=W(G,S)=N_G(S)(F)/T(F)
\]
for the relative Weyl group and fix a Whittaker datum for $G$.  All inductions
and Jacquet functors in this article are normalized.
We use the same notation for a character of a complex torus and its
differential.  For a smooth character $\eta$ of $T(F)$, write
$\widehat\eta:W_F\to{}^LT$ for a representative of its torus
Langlands parameter $\LLC_T(\eta)$.

We fix geometric Frobenius $\mathrm{Fr}_F$ and use the reciprocity
normalization $\operatorname{Art}_F(\varpi_F)=\mathrm{Fr}_F$ in
$W_F^{\mathrm{ab}}$.  Thus when $T$ is split, for an unramified character $\eta$ of $T(F)$
and $\mu\in X_*(T)=X^*(T^\vee)$,
\begin{equation}\label{eq:torus-reciprocity}
 \mu\bigl(\widehat\eta(\mathrm{Fr}_F)\bigr)
       =\eta\bigl(\mu(\varpi_F)\bigr),
\end{equation}
where $X_*(T)$ and $X^*(T^\vee)$ are the cocharacter lattice of $T$ and the character lattice 
of the dual torus $T^\vee$, respectively. 
The LLC for a general torus $T$ is taken with the same normalization;
see \cite[Section~3.3.1 and equation~(4.0.4)]{Hai14}.
Let $G^\vee$ be the complex dual group, with Borel pair
$T^\vee\subset B^\vee=T^\vee U^\vee$ dual to $(T,B)$: the simple
roots of $B^\vee$ correspond to the simple absolute coroots
determined by $B$.  Fix a compatible pinning and the associated Weil
action.  We write
\[
 B^{\vee,-}:=\overline{B^\vee}
\]
for the opposite Borel relative to $T^\vee$, so that
$B^\vee\cap B^{\vee,-}=T^\vee$.
We use $B^{\vee,-}$ in the transported Bernstein and geometric
Hecke presentations.
Indeed, the evaluation character attached to $\mu(\varpi_F^{-1})$
corresponds in \eqref{eq:torus-reciprocity} to $-\mu$ on $T^\vee$.
The same sign holds on the real unramified character spaces for a
nonsplit torus, by functoriality for cocharacters and norm maps.
Section~\ref{sec:hecke} specifies the transport of the whole Hecke and
geometric datum that accompanies this choice of Borel.

Let $\chi:T(F)\to\C^\times$ be smooth and regular, in the sense that
\[
  \Stab_W(\chi)=\{1\}.
\]
Write $\chi=\chi_u\chi_r$, where $\chi_u$ is unitary and
$\chi_r(t)=|\chi(t)|$.  For a reduced relative root $\alpha$,
let $M_\alpha$ be the Levi subgroup containing $T$ whose
relative roots are those proportional to $\alpha$.
The associated normalized co-rank-one induction is
$\Ind_{B\cap M_\alpha}^{M_\alpha}(\chi)$;
see \cite[Section~2.4]{Luo21}.
For each root line on which this induction is reducible,
choose the sign of its coroot $a$ so that
\begin{equation}\label{eq:positive-exponent}
  \chi_r\circ H_a=|\,\cdot\,|^{s_a},\qquad s_a>0,
\end{equation}
where $H_a:F^\times\to S(F)\subset T(F)$ is an integral
realization of the relative coroot direction.
For every reducibility line, the associated Weyl reflection $w_a$
fixes $\chi_u$; see \cite[Section~4]{Luo21}.
If the real exponent on that line were zero, $w_a$ would also fix
$\chi_r$, and hence $\chi$, contradicting regularity.
Thus the exponent is nonzero, and exactly one of the two choices
of sign for $a$ makes it positive.  This gives the orientation in
\eqref{eq:positive-exponent}.
If a relative coroot is
initially specified only as a rational cocharacter, choose a positive
integral multiple.  Replacing $H_a$ by its $m$-th power multiplies both
$s_a$ and the index $d_a$ below by $m$; the ratio $s_a/d_a$, the wall,
and its orientation are unchanged.

Put
\[
  \mathfrak a_T^*:=\Hom\bigl(T(F)/T(F)_{\mathrm{cpt}},\R\bigr).
\]
Here $T(F)_{\mathrm{cpt}}$ is the maximal compact subgroup of $T(F)$, $\mathfrak a_T=(\mathfrak a_T^*)^*$, and relative coroots are viewed
as elements of $\mathfrak a_T$ with its natural pairing with
$\mathfrak a_T^*$.  This convention for coroots is distinct from the
identification of the valuation lattice with cocharacters used below.

\begin{definition}[Reducibility arrangement]
\label{def:reducibility-arrangement}
Let $\cS(\chi)$ consist of the chosen coroot on each reducibility
line, oriented by \eqref{eq:positive-exponent}.
For $a\in\cS(\chi)$, put
\begin{align*}
 \mathcal H_a
   &:=\{x\in\mathfrak a_T^*:\langle x,a\rangle=0\},\\
 \mathcal H_a^+
   &:=\{x\in\mathfrak a_T^*:\langle x,a\rangle>0\}.
\end{align*}
The chambers of the reducibility arrangement are
\[
 \Ch(\chi):=\pi_0\left(
 \mathfrak a_T^*\setminus
 \bigcup_{a\in\cS(\chi)}\mathcal H_a
 \right),
\]
where $\pi_0$ denotes the set of connected components.
\end{definition}

We write
\begin{align*}
 X_{\mathrm{nr}}(T)
   &:=\Hom\bigl(T(F)/T(F)_{\mathrm{cpt}},\C^\times\bigr),\\
 X_{\mathrm{nr}}^+(T)
   &:=\Hom\bigl(T(F)/T(F)_{\mathrm{cpt}},\R^+\bigr)
\end{align*}
for the group of unramified characters of $T(F)$ and its
positive real-valued subgroup, respectively, where
$\R^+=(0,\infty)$ is viewed multiplicatively.
Pointwise exponentiation gives an isomorphism
\[
 \exp:\mathfrak a_T^*\xrightarrow{\sim}X_{\mathrm{nr}}^+(T),
 \qquad (\exp x)(h)=e^{x(h)}.
\]
Let $C^+$ be the dominant open Weyl chamber determined by $B$ and put
$C^-=w_0C^+$, where $w_0$ is the longest element of $W$.

For $a\in\cS(\chi)$, let $h_a$ be the primitive element on the corresponding
line in $T(F)/T(F)_{\mathrm{cpt}}$, oriented so that
\begin{equation}\label{eq:lattice-index}
  H_a(\varpi_F^{-1})\equiv h_a^{d_a}\pmod{T(F)_{\mathrm{cpt}}}
\end{equation}
for a positive integer $d_a$.  The rank-one labels and the formula for
$s_a/d_a$ are given in Section~\ref{sec:rank-one}.

The Bernstein torus attached to $\chi$ is the family
$\chi X_{\mathrm{nr}}(T)$ of its unramified twists;
see \cite[Section~1]{Sol23}.
Choose a unitary base point $\chi_0$ in this family, with the additional
normalization specified in Section~\ref{sec:hecke}.
  Taking $\chi_0$ as the identity
identifies this family with $X_{\mathrm{nr}}(T)$.
Its image under $\LLC_T$ is the corresponding dual Bernstein torus,
based at $\LLC_T(\chi_0)$.

Here and below, root identifications use the coordinate-algebra
isomorphism induced by $\LLC_T$.
For $h\in T(F)/T(F)_{\mathrm{cpt}}$, let $\kappa(h)$ be the
character of the dual Bernstein torus with $\LLC_T(\chi_0)$ as its
identity element, defined~by
\begin{equation}\label{eq:transported-coordinate}
 \kappa(h)\bigl(\LLC_T(\eta)\bigr)
       =\frac{\eta(h)}{\chi_0(h)},
\end{equation}
for all $\eta\in\chi_0X_{\mathrm{nr}}(T)$.
The quotient is computed using any representative of $h$ in $T(F)$
and is independent of that choice, since $\eta/\chi_0$ is unramified.
Changing the unitary base point multiplies $\kappa(h)$ by a constant
of modulus one, leaving its logarithmic differential and modulus
unchanged.
The geometric realization associates to the logarithmic differential
of $\kappa(h_a)$ a root on the same oriented line.  We denote that
\emph{geometric root} by $\iota(a)$; its positive scale is
computed in Lemma~\ref{lem:root-parameter-scale}.  It is a root of
\[
 H_\chi^\circ,\qquad
\text{where}\quad H_\chi:=Z_{G^\vee}\bigl(\widehat{\chi_u}(W_F)\bigr),
\]
relative to the maximal torus supplied by the principal-series datum.
Thus the root vectors in the theorem lie in $\Lie(H_\chi^\circ)$;
they need not lie in individual absolute root spaces of $G^\vee$.
The map $\iota$ preserves signs from the relative positive system for
$B$ to the geometric positive system for $B^{\vee,-}$.
Relative to $B^\vee$ its sign is the opposite one.
Its Weyl equivariance includes transport between the unitary
centralizers associated to Weyl-conjugate characters.
For a chamber
$\Gamma\in\Ch(\chi)$, let $\pi_\Gamma$ be the irreducible constituent
characterized by
\begin{equation}\label{eq:chamber-jacquet}
 r_B(\pi_\Gamma)=
 \bigoplus_{\substack{w\in W\\wC^+\subset\Gamma}}\chi^w,
 \qquad \chi^w=\chi\circ\Ad(w).
\end{equation}
The existence and uniqueness of these constituents follow from
\cite[Theorems~3.4 and~4.1]{Luo21}.

We also regard $\widehat\chi:W_F\to{}^LT$ as a parameter into ${}^LG$.
For an $L$-parameter $\phi:W_F\times\mathrm{SL}_2(\C)\to{}^LG$, define
\begin{equation}\label{eq:infinitesimal-convention}
 \widetilde\phi(w)=\phi\left(w,
 \begin{pmatrix}\|w\|^{1/2}&0\\0&\|w\|^{-1/2}\end{pmatrix}\right),
 \qquad
 N_\phi=d\phi\left(0,\begin{pmatrix}0&1\\0&0\end{pmatrix}\right),
\end{equation}
where $\|\mathrm{Fr}_F\|=q_F^{-1}$.
Thus $\Ad(\widetilde\phi(w))N_\phi=\|w\|N_\phi$, and in the
real logarithmic coordinates of Section~\ref{sec:hecke} this becomes
$[\sigma,N_\phi]=-\log(q_F)N_\phi$.
We use these conventions throughout; see \cite[Section~4]{Sol23}.

For a quasi-split group, an enhancement is an irreducible
representation $\rho$ of
\[
 \mathcal R_\phi:=
 \pi_0\!\left(
 Z_{G^\vee}\bigl(\phi(W_F\times\mathrm{SL}_2(\C))\bigr)
 /Z(G^\vee)^{W_F}\right).
\]
In the equivalent convention using the simply connected cover of the
derived dual group, quasi-split relevance requires the trivial central
character.  The displayed quotient incorporates that condition;
see \cite[Section~4]{Sol23}.

\begin{theorem}\label{thm:main}
For every $\Gamma\in\Ch(\chi)$, the principal-series LLC with the
Whittaker and torus normalizations specified above attaches to
$\pi_\Gamma$ an enhanced parameter $(\phi_\Gamma,\rho_\Gamma)$ which has a representative satisfying
\begin{equation}\label{eq:main-formula}
 \widetilde\phi_\Gamma=\widehat\chi,
 \qquad
 N_\Gamma\sim
 \sum_{\substack{a\in\cS(\chi)\\
                  \Gamma\subset\mathcal H_a^+}}
       X_{-\iota(a)},
 \qquad
 \rho_\Gamma=\triv.
\end{equation}
Here $N_\Gamma:=N_{\phi_\Gamma}$, and $X_{-\iota(a)}$ is an arbitrary
nonzero vector in the one-dimensional
$-\iota(a)$-root space.  Conjugacy in \eqref{eq:main-formula} is by
$Z_{G^\vee}\bigl(\widehat\chi(W_F)\bigr)^\circ$,
the identity component of the centralizer of the fixed
infinitesimal parameter.  All choices of nonzero root vectors
with the displayed support are conjugate under this torus.
The symbol $\triv$ means the one-dimensional trivial representation of
$\mathcal R_{\phi_\Gamma}$; this group need not itself be trivial.
\end{theorem}

The chamber containing $C^+$ corresponds to the unique irreducible
subrepresentation of $\Ind_B^G(\chi)$.  The theorem therefore gives its
parameter as follows.

\begin{corollary}[Parameter of the socle]\label{cor:socle}
There is a unique chamber $\Gamma_{\mathrm{soc}}\in\Ch(\chi)$ with
$C^+\subset\Gamma_{\mathrm{soc}}$, and
\[
  \soc\Ind_B^G(\chi)=\pi_{\Gamma_{\mathrm{soc}}}.
\]
Its enhanced parameter may be represented by
\begin{equation}\label{eq:socle-formula}
 \widetilde\phi_{\mathrm{soc}}=\widehat\chi,
 \qquad
 N_{\mathrm{soc}}\sim
 \sum_{\substack{a\in\cS(\chi)\\C^+\subset\mathcal H_a^+}}
 X_{-\iota(a)},
 \qquad
 \rho_{\mathrm{soc}}=\triv.
\end{equation}
Here $N_{\mathrm{soc}}:=N_{\phi_{\mathrm{soc}}}$.
\end{corollary}

We first recall the chamber description and identify the chamber of the
socle in Section~\ref{sec:chambers}.  Sections~\ref{sec:hecke}
and~\ref{sec:rank-one} pass to the graded Hecke algebra and identify
the roots that can occur in monodromy.  The calculation of standard-module
weights in Section~\ref{sec:fixed-locus} then determines the irreducible
module attached to each chamber in Section~\ref{sec:constituents}.
Section~\ref{sec:parameters} proves the main theorem and discusses the
socle, the conditions of Bella\"iche--Chenevier, and the case of
$\mathrm{GL}_2$.  Section~\ref{sec:unitary-example} treats the
unramified unitary group in three variables.

\section{Chambers, Jacquet modules, and the socle}\label{sec:chambers}

We recall the chamber parametrization by Rodier and Luo.  Rodier treated
split groups in \cite{Rod81}; the formulation below follows from Luo's
extension to generalized principal series.  Since $T$ is a torus, every
irreducible smooth representation of $T(F)$ is supercuspidal, and the
regularity condition in \cite{Luo21} becomes $\Stab_W(\chi)=\{1\}$.
We use the oriented reducibility coroots $\cS(\chi)$ and the
chambers $\Ch(\chi)$ of their hyperplane arrangement, as in
Definition~\ref{def:reducibility-arrangement}.
Changing the sign of a reducibility coroot leaves the arrangement
unchanged.

\begin{theorem}[Rodier; Luo]\label{thm:chambers}
The elements of $\cS(\chi)$ are linearly independent in
$\mathfrak a_T$.  The Jordan--H\"older constituents of
$I(\chi)=\Ind_B^G(\chi)$ are pairwise nonisomorphic and are parametrized by
$\Ch(\chi)$.  Their normalized Jacquet modules are given by
\eqref{eq:chamber-jacquet}; every displayed exponent occurs with
multiplicity one.
\end{theorem}

\begin{proof}
We apply \cite[Theorems~3.4 and~4.1]{Luo21} with $M=T$.
Every reduced relative root contributes a reflection in $W$,
so Luo's space ${}^0\mathfrak a_T^*$ is the span $V$ of the
relative roots; see \cite[Section~2.2]{Luo21}.
We have
\[
 \mathfrak a_T^*=V\oplus Z,
 \qquad Z=(\mathfrak a_T^*)^W.
\]
Every relative coroot vanishes on $Z$.  Hence
\[
 \mathcal H_a=(\mathcal H_a\cap V)\oplus Z,
 \qquad C^+=(C^+\cap V)+Z.
\]
Thus $\Gamma_V\mapsto\Gamma_V+Z$ identifies chambers with
$\Ch(\chi)$ and preserves the containment conditions in
\eqref{eq:chamber-jacquet}.
The induction and Jacquet functors in \cite{Luo21} are normalized.
Changing the signs of the reducibility coroots affects neither
the arrangement nor linear independence.
\end{proof}

Let $\overline B=T\overline U$ be the Borel opposite to $B$.  Conjugating the
Jacquet formula in Theorem~\ref{thm:chambers} by a representative of $w_0$ gives
\begin{equation}\label{eq:opposite-jacquet}
 r_{\overline B}(\pi_\Gamma)=
 \bigoplus_{\substack{v\in W\\vC^-\subset\Gamma}}\chi^v,
\end{equation}
again with multiplicity one.

\begin{proposition}\label{prop:socle-chamber}
The unique chamber containing $C^+$ parametrizes the unique irreducible
subrepresentation of $I(\chi)$.
\end{proposition}

\begin{proof}
Every hyperplane $\mathcal H_a$ is a relative root hyperplane, so the open
Weyl chamber $C^+$ is disjoint from it.  Since $C^+$ is connected, it lies in
a unique chamber $\Gamma_{\mathrm{soc}}$.

Normalized Frobenius reciprocity gives
\[
 \Hom_{G(F)}(\pi_\Gamma,I(\chi))
 \cong \Hom_{T(F)}(r_B(\pi_\Gamma),\chi).
\]
By \eqref{eq:chamber-jacquet} and regularity, the right-hand side is
one-dimensional precisely when $C^+\subset\Gamma$, and is zero otherwise.
Thus $\pi_{\Gamma_{\mathrm{soc}}}$ embeds in $I(\chi)$.  Every irreducible
subrepresentation is one of the pairwise nonisomorphic constituents in
Theorem~\ref{thm:chambers}, so no other constituent embeds.  Hence the socle is
irreducible and equals $\pi_{\Gamma_{\mathrm{soc}}}$.
\end{proof}

In the next section, Bernstein's second adjointness identifies the
exponents in \eqref{eq:opposite-jacquet} with weights of the Hecke module
attached to $\pi_\Gamma$.

\section{Hecke algebras and geometric parameters}
\label{sec:hecke}

Let $\mathfrak s=[T,\chi]_G$ be the Bernstein inertial class of
$(T,\chi)$, and write $\operatorname{Rep}(G(F))^{\mathfrak s}$ for
the corresponding block.  We write
\[
  T_{\mathfrak s}=\chi_0X_{\mathrm{nr}}(T)=\chi X_{\mathrm{nr}}(T)
\]
for the Bernstein torus introduced above.
We choose and fix a unitary base point $\chi_0$ obtained via the torus
LLC from a bounded base point chosen as in
\cite[Sections~4--5]{Sol23}.
The existence of such a choice follows from
\cite[Propositions~3.9(b) and~3.14(b)]{AMS17}.
As in \cite[Section~5]{Sol23}, $\chi_0$ is fixed by the reflection
subgroup.  Hence the identification
\[
 X_{\mathrm{nr}}(T)\xrightarrow{\sim}T_{\mathfrak s},
 \qquad \nu\longmapsto\chi_0\nu,
\]
is equivariant for this subgroup.
All rank-one Bernstein coordinates below are defined relative to this
same base point, as in \eqref{eq:transported-coordinate}.
Let
\[
 P_T=\operatorname{c-Ind}_{T(F)_{\mathrm{cpt}}}^{T(F)}
       (\chi_0|_{T(F)_{\mathrm{cpt}}}),
 \qquad
 \Pi_{\mathfrak s}=\Ind_B^G(P_T).
\]
Set
$\cH_{\mathfrak s}=\operatorname{End}_{G(F)}(\Pi_{\mathfrak s})$.
The coordinate algebra $\cO(T_{\mathfrak s})$ is a maximal
commutative subalgebra of $\cH_{\mathfrak s}$;
see \cite[Section~1]{Sol23}.
The representation $\Pi_{\mathfrak s}$ is a progenerator of the block.
By Morita theory, the functor
\begin{align*}
 M:\operatorname{Rep}(G(F))^{\mathfrak s}
   &\xrightarrow{\ \sim\ }
     \operatorname{Mod}(\cH_{\mathfrak s}^{\mathrm{op}}),\\
 \pi&\longmapsto\Hom_{G(F)}(\Pi_{\mathfrak s},\pi)
\end{align*}
is an equivalence, where $\operatorname{Mod}$ denotes the category
of left modules; see \cite[Section~1, equation~(1-1)]{Sol23}.
The equivalence restricts to finite-length representations in the block
and finite-dimensional modules.  All modules considered below belong
to these finite-length categories.

For an affine complex variety $Y$ and a finite-dimensional module $V$
over an algebra containing $\cO(Y)$, write $V_t$ for its generalized
weight space at $t\in Y$:
\[
 V_t=\{v\in V:\mathfrak m_t^n v=0\text{ for some }n\geq1\},
 \qquad \mathfrak m_t=\{f\in\cO(Y):f(t)=0\}.
\]
Geometrically, the underlying $\cO(Y)$-module defines a coherent
sheaf $\mathcal F$ of finite length on $Y$, and $V_t$ is naturally
identified with its stalk $\mathcal F_t$.  In particular,
\[
 \dim_\C V_t
   =\operatorname{length}_{\cO_{Y,t}}(\mathcal F_t).
\]
We use this notation for both torus coordinate algebras and polynomial
algebras on Lie algebras.

For smooth representations $\tau$ of $T(F)$ and $\pi$ of $G(F)$,
Bernstein's second adjointness gives the natural isomorphism
\[
 \Hom_{G(F)}\bigl(\Ind_B^G(\tau),\pi\bigr)
 \simeq\Hom_{T(F)}\bigl(\tau,r_{\overline B}(\pi)\bigr),
\]
where induction and Jacquet restriction are normalized; see
\cite[Section~2.1, equation~(2-1)]{Sol23}.
Taking $\tau=P_T$ gives
\[
 M(\pi)\simeq
 \Hom_{T(F)}\bigl(P_T,r_{\overline B}(\pi)\bigr).
\]
Thus the generalized $\cO(T_{\mathfrak s})$-weights of $M(\pi)$
are precisely the exponents in the $T_{\mathfrak s}$-summand of
$r_{\overline B}(\pi)$, with their multiplicities and without
inverting the characters.  Since $\Pi_{\mathfrak s}$ generates
the block, $M(\pi_\Gamma)\ne0$.  Combining this with
\eqref{eq:opposite-jacquet}, we find that for every chamber
$\Gamma$ there exists $v\in W$ such that
\begin{equation}\label{eq:selected-v}
 vC^-\subset\Gamma,\qquad
 \chi^v\in T_{\mathfrak s},\qquad
 \dim_\C M(\pi_\Gamma)_{\chi^v}=1.
\end{equation}

Let $\mathfrak s^\vee$ be the inertial class of $\widehat\chi$, namely
its orbit under unramified twists as a torus parameter.  The torus local
Langlands correspondence identifies $T_{\mathfrak s}$ with the parameter
torus of this class, which we denote by $T_{\mathfrak s}^\vee$.
Put $W_{\mathfrak s}=\Stab_W(T_{\mathfrak s})$; torus LLC identifies it
with the extended Weyl group of the dual inertial datum.  In the transported
Bernstein presentation specified below, the Whittaker comparison is
\begin{equation}\label{eq:affine-comparison}
 \psi_{\mathfrak s}:\cH_{\mathfrak s}^{\mathrm{op}}
 \xrightarrow{\ \sim\ }
 \cH(\mathfrak s^\vee,q_F^{1/2}).
\end{equation}
We use the comparison construction of \cite[Theorems~2.7 and~5.4]{Sol23}
and the Galois-side Hecke algebras of \cite{AMS17}, with the following
explicit change of dual coordinates.

Let $C^\vee$ be the Chevalley involution
associated with the chosen Weil-stable pinning.  It acts by inversion
on $T^\vee$, sends $B^\vee$ to $B^{\vee,-}$, and extends to ${}^LG$
by fixing $W_F$; see \cite[Section~1]{Kal13}.

With torus map $C^\vee_*\circ\LLC_T$, the positive primitive
Bernstein coordinates correspond to roots positive for $B^\vee$,
and the two affine-Hecke labels agree under this correspondence;
see \cite[Lemma~5.3]{Sol23}.
By \cite[Theorem~5.4]{Sol23}, these identifications extend to an
isomorphism from the representation-side opposite Hecke algebra
to the dual affine Hecke algebra.
Transport the resulting dual datum, including its inertial class,
coordinate algebra, and Borel, through $C^\vee$.
Since $(C^\vee_*)^2=1$, the resulting comparison $\psi_{\mathfrak s}$ in
\eqref{eq:affine-comparison} satisfies
\[
 \psi_{\mathfrak s}(f)=f\circ\LLC_T^{-1}
 \quad(f\in\cO(T_{\mathfrak s})),
 \qquad \psi_{\mathfrak s}(N_w)=N_{w^{-1}}.
\]
The first identity uses the full torus correspondence, including
central and ramified characters.  The target Borel is $B^{\vee,-}$,
with the same Hecke labels, and the root coordinates agree with
\eqref{eq:transported-coordinate}.

Under the same transport, nilpotent elements are carried by
$dC^\vee$, and enhancements by the induced isomorphisms of
component groups.  The geometric parametrization is compatible
with this action; see \cite[proof of Theorem~7.1]{Sol23}.
The transport introduces no signs on the normalized finite-Hecke
generators, so it preserves the finite-Hecke Steinberg character
and the Whittaker normalization.

Fix a chamber $\Gamma\in\Ch(\chi)$ and choose $v\in W$ satisfying
\eqref{eq:selected-v}.  Keep these choices fixed throughout the reduction
below.
Transport $M(\pi_\Gamma)$ through \eqref{eq:affine-comparison}, and
denote the resulting module over the affine Hecke algebra
$\cH(\mathfrak s^\vee,q_F^{1/2})$ by $V_{\mathrm{aff}}$.
The restriction of \eqref{eq:affine-comparison} to the coordinate
algebra, together with second adjointness and \eqref{eq:opposite-jacquet},
gives the generalized weight set
\begin{equation}\label{eq:affine-weight-set}
 \wt(V_{\mathrm{aff}})=
 \left\{\LLC_T(\chi^w):
 w\in W,\ wC^-\subset\Gamma,\ \chi^w\in T_{\mathfrak s}\right\}.
\end{equation}
Each weight occurs with multiplicity one.  Put
\[
 t_v:=\LLC_T(\chi^v)=u\exp(\sigma_v)
\]
for its unitary--positive polar decomposition.  The module
$V_{\mathrm{aff}}$ is irreducible, and \eqref{eq:affine-weight-set} shows
that its weights lie in $W_{\mathfrak s}t_v$.  The
$W_{\mathfrak s}$-stabilizer of $t_v$ is trivial by the
$W$-equivariance of the torus LLC and the
regularity of $\chi$.

Here $u$ denotes the bounded torus parameter on the whole Weil group;
its dependence on the fixed $v$ is suppressed in the notation.
The geometric group and Borel in the reduction at $u$ are
\[
 G_u^\vee=Z_{G^\vee}\bigl(u(W_F)\bigr),\qquad
 B_u^{\vee,-}=(G_u^{\vee,\circ}\cap B^{\vee,-})^\circ,
\]
where $G_u^{\vee,\circ}$ denotes the identity component of $G_u^\vee$.
Let $\cH_u$ be the extended affine Hecke algebra obtained by reduction
at the unitary part $u$, and let $\mathbb H_u$ be its corresponding
extended graded Hecke algebra.  Its extended Weyl group is
$W_{\mathfrak s,u}=\Stab_{W_{\mathfrak s}}(u)$.
For either kind of Hecke algebra $\mathcal A$, the notation
$\operatorname{Mod}(\mathcal A)_\Omega$ means the full subcategory
of finite-dimensional modules whose generalized weights lie in $\Omega$.
View $X_{\mathrm{nr}}^+(T)$ on the dual side
through the torus LLC, and let $\mathfrak t_{u,\R}$ be its
real tangent space.
Unitary-part reduction and exponential reduction give equivalences
\begin{align}\label{eq:reductions}
 \operatorname{Mod}(\cH_u)_{X_{\mathrm{nr}}^+(T)u}
   &\xrightarrow[\sim]{\operatorname{ind}_u}
   \operatorname{Mod}(\cH(\mathfrak s^\vee,q_F^{1/2}))_{W_{\mathfrak s}X_{\mathrm{nr}}^+(T)u},\\
 \operatorname{Mod}(\mathbb H_u)_{\mathfrak t_{u,\R}}
   &\xrightarrow[\sim]{\exp_{u*}}
   \operatorname{Mod}(\cH_u)_{X_{\mathrm{nr}}^+(T)u},
   \qquad \sigma\longmapsto u\exp(\sigma).\nonumber
\end{align}
Both equivalences preserve irreducibility; see
\cite[Theorem~6.1]{Sol23} and
\cite[Theorem~6.2]{Sol23}, respectively.
The inverse of $\operatorname{ind}_u$ takes the sum of the
generalized weight spaces whose weights lie in $X_{\mathrm{nr}}^+(T)u$;
it retains the selected $t_v$-weight space.
The functor $\exp_{u*}$ preserves the underlying vector space and takes
the $\sigma_v$-weight space to the $t_v$-weight space, with the same
multiplicity; see \cite[Theorem~6.2(b)--(c)]{Sol23}.
Since all weights of
$V_{\mathrm{aff}}$ lie in
$W_{\mathfrak s}t_v\subset W_{\mathfrak s}X_{\mathrm{nr}}^+(T)u$,
the inverse reductions are
applicable to $V_{\mathrm{aff}}$.  We obtain an irreducible
$\mathbb H_u$-module $V_{\mathrm{gr}}$ with generalized weight set
\[
 \wt(V_{\mathrm{gr}})=
 \left\{\sigma\in\mathfrak t_{u,\R}:
 u\exp(\sigma)\in\wt(V_{\mathrm{aff}})\right\}.
\]
Every weight occurs with multiplicity one.  In particular, the selected
weight $t_v=u\exp(\sigma_v)$ gives
\begin{equation}\label{eq:selected-sigma-weight}
  \dim (V_{\mathrm{gr}})_{\sigma_v}=1.
\end{equation}

For a nonsplit torus, $u\exp(\sigma_v)$ is first interpreted in the
dual Bernstein torus.  We identify $\mathfrak t_{u,\R}$ with
$\R\otimes_\Z X_*(A)$ by differentiating the canonical isogeny
\[
 A:=(T^{\vee,W_F})^\circ\longrightarrow
 \bigl((T^{\vee,I_F})_{\mathrm{Fr}_F}\bigr)^\circ
 \simeq X_{\mathrm{nr}}(T);
\]
see \cite[equation~(4-3)]{Sol23}.
The Frobenius subscript denotes coinvariants.  Thus $\sigma_v$ is the
logarithmic lift of the positive part to $\Lie(A)$.
The torus $A$ is maximal in $G_u^{\vee,\circ}$.
For a root $\beta$ relative to $A$, write
$\mathfrak g_\beta=\Lie(G_u^{\vee,\circ})_\beta$.
We view these root spaces inside $\Lie(G^\vee)$; they need not be
absolute root spaces of $G^\vee$.

The next lemma describes the centralizer of a regular parameter. 

\begin{lemma} 
\label{lem:full-centralizer}
With $H=G_u^\vee$ and $M_u:=Z_H(A)$, one has
\begin{equation}\label{eq:full-centralizer-corrected}
 Z_H(\sigma_v)=M_u=(T^\vee)^{W_F}=A Z(G^\vee)^{W_F}.
\end{equation}
The identity component of $M_u$ is $A$, and $\exp(\sigma_v)$ is
regular semisimple in $H^\circ$.
The stabilizer of $\sigma_v$ in $N_H(A)/M_u$ is trivial.
\end{lemma}

\begin{proof}
An element of $W_{\mathfrak s,u}$ fixing $\sigma_v$ fixes
$u\exp(\sigma_v)=t_v$, whose stabilizer is trivial by regularity.
In particular, no root of $H^\circ$ vanishes on $\sigma_v$, and
$Z_{H^\circ}(\sigma_v)=A$.  These roots take real values on $\sigma_v$,
so $\alpha(\exp\sigma_v)=1$ also forces $\alpha(\sigma_v)=0$.
This proves the stated regular semisimplicity.

For the pinned Weil action on $G^\vee$, every absolute positive
root restricts nontrivially to $A$: its average over the finite image
of the Weil action is a nonzero sum of positive roots.  Hence
$Z_{G^\vee}(A)=T^\vee$, and intersecting with $H$ gives
$M_u=(T^\vee)^{W_F}$.
An element of $Z_H(\sigma_v)$ normalizes its connected centralizer
$A$, hence also $T^\vee$.  Its Weyl image is Weil invariant and
therefore represents an element of the relative Weyl group; see
\cite[equation~(5-2)]{Sol23}.  Since the element centralizes both
$u$ and $\sigma_v$, that Weyl image fixes the torus parameter $t_v$.
The Weyl equivariance of torus LLC and regularity of $\chi^v$
force the image to be the identity.  The element therefore belongs
to $T^\vee\cap H=M_u$.  The reverse inclusion is immediate.
This also proves $M_u^\circ=A$ and that the stabilizer of
$\sigma_v$ in $N_H(A)/M_u$ is trivial.

It remains to identify the components of $M_u$.  Let $T^\vee_{\rm ad}$
be the image of $T^\vee$ in the adjoint dual group.  Its character
lattice has a basis of simple roots permuted by the pinned Weil
action.  The coinvariant lattice is consequently torsion free, so
$(T^\vee_{\rm ad})^{W_F}$ is connected.  Exactness of invariants
for the finite Weil action on the Lie algebras shows that
$A\to(T^\vee_{\rm ad})^{W_F}$ has surjective differential and
hence is surjective as a map of complex tori.  Every element of
$M_u$ thus differs from an element of $A$ by an element of
$Z(G^\vee)^{W_F}$, proving $M_u=A Z(G^\vee)^{W_F}$.
\end{proof}

For the geometric graded algebra the principal-series quasi-Levi is
$M_u$, with the trivial equivariant local system on the zero orbit.
The associated quasi-parabolic is $B_u^{\vee,-} Z(G^\vee)^{W_F}$.
Put
\begin{equation}\label{eq:quasi-levi-flag}
 W_u=N_H(A)/M_u.
\end{equation}
Since $M_u=Z_H(A)$, conjugation induces a faithful action of $W_u$ on $A$.
These are the quasi-Levi and parabolic conventions of
\cite[Section~1, equation~(1.1)]{AMS17}.  We record how their Hecke
data agree with those of the reduction above.

There is a natural identification $W_u\cong W_{\mathfrak s,u}$.
Indeed, an element of $N_H(A)$ normalizes
$T^\vee=Z_{G^\vee}(A)$ and its Weyl image fixes the class of $u$.
The kernel is $H\cap T^\vee=M_u$.  Conversely, a Weyl element
fixing the torus-parameter class of $u$ has a representative which,
after multiplication by an element of $T^\vee$, centralizes $u$.
That representative normalizes $A$.  Here we use the Weyl
equivariance of torus LLC as in \cite[Lemma~5.1]{Sol23}.
Let $\Gamma_{u>0}\subset W_u$ be the subgroup preserving the positive
roots determined by $B_u^{\vee,-}$.  Then
\[
 W_u=W(H^\circ,A)\rtimes\Gamma_{u>0},\qquad
 \Gamma_{u>0}\cong H/(H^\circ M_u).
\]
The semidirect-product decomposition is the one in
\cite[Lemma~6.3]{Sol23}.
The second identification follows by conjugating a Borel subgroup
and a maximal torus inside $H^\circ$.  Thus the finite extension uses
the components modulo $H^\circ M_u$; the central components already
contained in $M_u$ do not give additional Weyl elements.

The coefficient system is the constant sheaf on a point with
trivial $M_u$-action.  For every normalizer element its pullback is
identified with itself by the identity on the stalk.  These maps
compose without scalar factors, so the associated $2$-cocycle is
trivial; compare \cite[proof of Theorem~4.2]{Sol23}.  The resulting
finite extension acts by its natural action on $\Lie(A)^*$ and by
permuting the reflection generators, as in
\cite[Proposition~1.1 and equation~(1.9)]{AMS17}.  The polynomial
algebra is the full $\operatorname{Sym}(\Lie(A)^*)$, including the
central directions.  With the root and parameter identifications
of Lemma~\ref{lem:root-parameter-scale}, the reduced algebra
$\mathbb H_u$ is the geometric graded Hecke algebra attached to
$H$, the quasi-Levi $M_u$, and the trivial local system on the zero
orbit, with parameter $\log(q_F)/2$; see
\cite[Section~6.2, equation~(6-4)]{Sol23}.
Its irreducibles have
geometric parameters
\begin{equation}\label{eq:geometric-triple}
 (\sigma,N,\rho),\qquad [\sigma,N]=-\log(q_F)N,
 \qquad \rho\in\Irr\bigl(\pi_0(Z_H(\sigma,N))\bigr),
\end{equation}
where
\[
 Z_H(\sigma,N):=Z_H(\sigma)\cap Z_H(N)
\]
is the simultaneous centralizer of $\sigma$ and $N$.
The enhancement has cuspidal quasi-support $(M_u,0,\triv)$.
In particular, the action of $Z(G^\vee)^{W_F}$
is trivial in the chosen principal-series datum.  The Springer condition
for $H^\circ$ alone would not specify that action.  Section~\ref{sec:fixed-locus}
computes eligibility directly with the quasi-Levi flag variety.

The polynomial central character of the module attached to
$(\sigma,N,\rho)$ is $W_{\mathfrak s,u}\sigma$; see
\cite[Proposition~1.4(b)]{AMS17} and
\cite[Section~6.2, equation~(6-5)]{Sol23}.
Since that of $V_{\mathrm{gr}}$ is $W_{\mathfrak s,u}\sigma_v$,
we may choose its parameter in the form
\[
 (\sigma_v,N,\rho),
\]
by conjugating the entire triple in $H$.

\section{Rank-one reducibility and monodromy}
\label{sec:rank-one}

We compare the rank-one Bernstein coordinates with the roots of
$G_u^{\vee,\circ}$ to determine the $-\log(q_F)$ adjoint eigenspace in
\eqref{eq:geometric-triple}.

Fix $a\in\cS(\chi)$, and let $z_a=\kappa(h_a)$ be the rank-one
Bernstein coordinate determined by the unitary base point $\chi_0$
fixed in Section~\ref{sec:hecke}.
Equations~\eqref{eq:positive-exponent} and
\eqref{eq:lattice-index} give
\begin{equation}\label{eq:z-modulus}
 |z_a(\chi)|=\chi_r(h_a)=q_F^{s_a/d_a}>1.
\end{equation}
Let $\lambda_a\geq\lambda_a^*\geq0$ be the two rank-one
affine-Hecke labels in this primitive coordinate, and put
\[
 Q_a=q_F^{(\lambda_a+\lambda_a^*)/2},\qquad
 Q_a^*=q_F^{(\lambda_a-\lambda_a^*)/2}.
\]
In the coordinate oriented by \eqref{eq:z-modulus}, a rank-one
Macdonald $c$-function and the corresponding Harish-Chandra
$\mu$-function, the latter up to a nonzero constant, are
\begin{equation}\label{eq:rank-one-mu}
 c_a(z)=\frac{(1-Q_a^{-1}z)(1+(Q_a^*)^{-1}z)}{1-z^2},
 \qquad
 \mu_a(z)=\frac{1}{c_a(z)c_a(z^{-1})}.
\end{equation}
This is the usual $c$-function expressed in the inverse root coordinate;
compare \cite[Section~8, equation~(8.2)]{DO08}, with the label
normalization of \cite[Section~1]{Sol23}.  The nonunitary
reducibility points are poles of $\mu_a$, rather than poles of $c_a$.
On the side $|z|>1$ they are precisely
\begin{equation}\label{eq:rank-one-poles}
 z=Q_a,
 \qquad
 z=-Q_a^*\quad\text{if }\lambda_a>\lambda_a^*.
\end{equation}
Indeed, away from $z=\pm1$ the only zeros of the two numerator
factors are $Q_a$ and $-Q_a^*$; the zeros contributed by $c_a(z^{-1})$
are their inverses, which have modulus less than one when they are
nonunitary.  At $\lambda_a=\lambda_a^*$ the factor $1+z$ cancels
from $c_a(z)$, so the apparent negative point $z=-1$ contributes no
pole.  Thus the unitary sign $e_a=z_a(\chi)/|z_a(\chi)|$ is $1$ or
$-1$, and the regular rank-one reducibility criterion gives
\begin{equation}\label{eq:rank-one-exponents}
 \frac{s_a}{d_a}=b_a,
 \qquad b_a:=\frac{\lambda_a+e_a\lambda_a^*}{2}>0.
\end{equation}

For the fixed choice of $v$ in \eqref{eq:selected-v}, transport the coordinate
$z_a$ and its unitary base point by $v^{-1}$, and denote the result by
$z_{a,v}$.  Its logarithmic differential is
\[
 \ell_{a,v}=d\log z_{a,v}\in\Lie(A)^*.
\]
The differential of the torus LLC is understood in this notation.
Because the unitary base point has modulus one,
\begin{equation}\label{eq:affine-log-value}
 \ell_{a,v}(\sigma_v)=\log|z_{a,v}(t_v)|=b_a\log(q_F).
\end{equation}
The corresponding root of $G_u^{\vee,\circ}$ is $v^{-1}\iota(a)$,
where the Weyl action also transports the unitary centralizer and
its parameter data.

\begin{lemma}
\label{lem:root-parameter-scale}
The root lines with positive graded Hecke parameter in the unitary
reduction are in bijection with the root lines of $G_u^{\vee,\circ}$.
On the line of $a\in\cS(\chi)$, the logarithmic differential satisfies
\begin{equation}\label{eq:root-parameter-scale}
 \ell_{a,v}=b_a\,v^{-1}\iota(a).
\end{equation}
The same comparison holds on every root line with positive graded Hecke
parameter, whether or not it is a reducibility line at $t_v$.
\end{lemma}

\begin{proof}
Let $z$ be a Bernstein coordinate in the unitary reduction, with labels
$\lambda,\lambda^*$ and unitary part $e\in\{1,-1\}$.  After exponential
reduction, its logarithmic differential $\ell=d\log z$ has graded Hecke
parameter
\[
 k_\ell^u=\tfrac12(\lambda+e\lambda^*)\log(q_F).
\]
The roots with $k_\ell^u>0$ form the root system of the graded
Hecke factor of $\mathbb H_u$.  Its finite extension by
$\Gamma_{u>0}$ uses the action described in Section~\ref{sec:hecke};
see \cite[Lemma~6.3]{Sol23}.

Put $J=Z_{G^\vee}(u(I_F))$.
Let $\gamma\in R(J^\circ,A)_{\mathrm{red}}$ be the reduced
restriction on the line under consideration.  Thus $\gamma$ is
chosen before imposing the bounded Frobenius condition; the actual
root of $G_u^{\vee,\circ}$ is denoted by $\delta$.
The root character in the affine root datum pulls back to
$m_\gamma\gamma$ on $A$, for a positive integer $m_\gamma$.
Its logarithmic differential is therefore $\ell_0=m_\gamma\gamma$;
see \cite[Definition~3.11]{AMS17} and
\cite[equation~(4-3)]{Sol23}.
Let $\lambda_0,\lambda_0^*$ be its labels.  The chosen coordinate has
logarithmic differential $\ell=\ell_0$ or $\ell_0/2$, the latter when
the root character is divisible by $2$ in the Bernstein character lattice.
We first compare roots before this replacement.

The principal-series cuspidal orbit is zero, so each root of
$G_u^{\vee,\circ}$ has geometric Hecke label $2$.
When the unitary part of the original coordinate is $1$, one has
\[
 \lambda_0+\lambda_0^*=2m_\gamma,
 \qquad \delta=\gamma,
 \qquad
 \ell_0=\tfrac12(\lambda_0+\lambda_0^*)\delta;
\]
see \cite[Proposition~3.14, equation~(3.31)]{AMS17}.
When that unitary part is $-1$, one has
\[
 \lambda_0-\lambda_0^*=
 \begin{cases}
 2m_\gamma,&\delta=\gamma,\\
 m_\gamma,&\delta=2\gamma.
 \end{cases}
\]
See \cite[Proposition~3.14, equation~(3.33)]{AMS17}.
In either case,
$\ell_0=\tfrac12(\lambda_0-\lambda_0^*)\delta$.
These are the cases $\delta=\gamma$ and $\delta=2\gamma$ in
\cite[equations~(2.18)--(2.20)]{AMS17}.
If neither $\gamma$ nor $2\gamma$ is a root of $G_u^{\vee,\circ}$,
the graded Hecke parameter is zero.

When $\ell=\ell_0/2$, the labels must be changed as well.
In this case $\lambda_0=\lambda_0^*=m_\gamma$, and the labels in
the new coordinate are $(\lambda,\lambda^*)=(m_\gamma,0)$;
see \cite[equations~(4-4)--(4-6) and Section~5, before
Lemma~5.3]{Sol23}.  Both unitary signs of the new coordinate map to
unitary part $1$ in the original coordinate.
Thus $\delta=\gamma$ and
\[
 \ell=\frac{m_\gamma}{2}\gamma
      =\frac{\lambda+e\lambda^*}{2}\delta
      \qquad(e=\pm1).
\]
All labels below refer to the chosen Bernstein coordinate.

For each root $\delta$ of $G_u^{\vee,\circ}$, we have $\ell=b\delta$, with
$b=(\lambda+e\lambda^*)/2>0$.  The graded cross relation has
coefficient $k_\ell^u/\ell=\log(q_F)/\delta$.  Thus the graded Hecke
parameter attached to $\delta$ is $\log(q_F)$, corresponding to the scalar
$\log(q_F)/2$ in the geometric realization; see
\cite[Lemma~6.4 and equation~(6-4)]{Sol23}.
The comparison preserves the polynomial algebra and its weights.
Applied to the oriented line of $a$, it proves
\eqref{eq:root-parameter-scale} and the asserted bijection.
\end{proof}

Combining \eqref{eq:affine-log-value} with
\eqref{eq:root-parameter-scale} yields
\begin{equation}\label{eq:positive-graded-root}
 [\sigma_v,X]=\log(q_F)X
 \qquad(0\ne X\in\mathfrak g_{v^{-1}\iota(a)}).
\end{equation}
Consequently, with $\beta_{a,v}=-v^{-1}\iota(a)$,
\begin{equation}\label{eq:monodromy-root-sign}
 [\sigma_v,Y]=-\log(q_F)Y
 \qquad(0\ne Y\in\mathfrak g_{\beta_{a,v}}).
\end{equation}
Both root spaces are those of $G_u^{\vee,\circ}$ relative to $A$.

\begin{proposition}\label{prop:monodromy-eigenspace}
The roots $\beta$ of $G_u^{\vee,\circ}$ satisfying
$\beta(\sigma_v)=-\log(q_F)$ are exactly
\[
 \Sigma_v=\{\beta_{a,v}:a\in\cS(\chi)\}.
\]
They are linearly independent, and
\begin{equation}\label{eq:full-monodromy-space}
 \{Y\in\Lie(G_u^{\vee,\circ}):[\sigma_v,Y]=-\log(q_F)Y\}
   =\bigoplus_{a\in\cS(\chi)}\mathfrak g_{\beta_{a,v}}.
\end{equation}
With positivity defined by
$B_u^{\vee,-}=(G_u^{\vee,\circ}\cap B^{\vee,-})^\circ$,
\begin{equation}\label{eq:positivity-chamber}
 \beta_{a,v}>0
 \quad\Longleftrightarrow\quad
 \Gamma\subset\mathcal H_a^+.
\end{equation}
\end{proposition}

\begin{proof}
Equation~\eqref{eq:monodromy-root-sign} supplies every root displayed
in $\Sigma_v$.  Conversely, suppose
$\beta(\sigma_v)=-\log(q_F)$ and put $\delta=-\beta$.
By Lemma~\ref{lem:root-parameter-scale}, there is a Bernstein coordinate
$z$ with logarithmic differential $\ell=b\delta$ and graded Hecke
parameter $b\log(q_F)$, where
$b=(\lambda+e\lambda^*)/2>0$ and $e\in\{1,-1\}$ is its
unitary part.  Its value at $t_v$ is therefore
\[
 z(t_v)=e\exp(\ell(\sigma_v))=e q_F^b.
\]
For $e=1$ this is the first pole of \eqref{eq:rank-one-poles}; for
$e=-1$ the inequality $b>0$ implies $\lambda>\lambda^*$ and it is
the second pole.  Thus the associated co-rank-one induction is
reducible.  Transporting the line back by $v$ gives a unique
$a\in\cS(\chi)$: its prescribed orientation is forced by
$\ell(\sigma_v)>0$.  The same root comparison then gives
$\beta=\beta_{a,v}$.  This proves the equality of root sets.

The root-space decomposition proves
\eqref{eq:full-monodromy-space}; its toral summand has eigenvalue
zero and hence contributes nothing.  Linear independence follows
from Theorem~\ref{thm:chambers}: the correspondence on root lines is the
linear differential of the torus LLC composed with the Weyl action,
and each individual root is changed only by a nonzero scale and a
sign.

To determine the chamber side, Weyl equivariance gives
\[
 \beta_{a,v}=-v^{-1}\iota(a)=-\iota(v^{-1}a).
\]
The sign convention in \eqref{eq:transported-coordinate} is
positive from $B$ to $B^{\vee,-}$.  Hence
$\beta_{a,v}$ is positive for $B_u^{\vee,-}$ exactly when $v^{-1}a$
is negative in the relative system.  Since $C^-$ is
anti-dominant, this is equivalent to
$vC^-\subset\mathcal H_a^+$.  Finally $vC^-\subset\Gamma$, and
$\Gamma$ is connected and disjoint from $\mathcal H_a$, so the
last condition is equivalent to
$\Gamma\subset\mathcal H_a^+$.  This proves
\eqref{eq:positivity-chamber}.
\end{proof}

For the fixed choices of $\Gamma$ and $v$, put
\begin{equation}\label{eq:P-definition}
 \Sigma_v^+=
 \{\beta_{a,v}\in\Sigma_v:\Gamma\subset\mathcal H_a^+\}.
\end{equation}
By Proposition~\ref{prop:monodromy-eigenspace}, this is precisely
the subset of $\Sigma_v$ positive for $B_u^{\vee,-}$.

\section{Weights of geometric standard modules}
\label{sec:fixed-locus}

We compute weight multiplicities using the quasi-Levi construction
of \cite[Section~1]{AMS17} and \cite[Sections~1 and~3]{Sol21}.

For the moment let $H$ be a possibly disconnected complex reductive
group, let $A\subset H^\circ$ be a maximal torus, and assume
\[
 M:=Z_H(A)=A Z(H).
\]
Let $B^\circ=A U^\circ$ be a Borel subgroup of $H^\circ$.
The quasi-Levi $M$ has associated quasi-parabolic
\begin{equation}\label{eq:quasi-parabolic-choice}
 P=M U^\circ=B^\circ Z(H),
 \qquad P\cap H^\circ=B^\circ.
\end{equation}
Equip the zero orbit with the constant local system on which $M$
acts trivially.  Write $N_H(P,M,\triv)$ for the subgroup preserving
$P$, $M$, and the isomorphism class of this equivariant local system.
Since the local system is trivial,
$N_H(P,M,\triv)=N_H(P)\cap N_H(M)$.
Assume
\begin{equation}\label{eq:borel-normalization}
 H=N_H(P,M,\triv)H^\circ.
\end{equation}
This is \cite[Condition~1.2]{Sol21}.
Put $W_H=N_H(A)/M=N_H(A)/Z_H(A)$.
Conjugation induces a faithful action of $W_H$ on $A$.
Let $r\in\C$, let $y\in\Lie(H^\circ)$ be nilpotent, and let
$\sigma\in\Lie(A)$ be semisimple, with
\[
 [\sigma,y]=2ry,\qquad Z_H(\sigma)=M,
\]
and suppose $\exp(\sigma)$ is regular semisimple in $H^\circ$.
Define
\[
 \cP_y=\{gP\in H/P:
                \Ad(g^{-1})y\in\Lie(U^\circ)\}.
\]
For nilpotent $y$, membership in $\Lie(B^\circ)$ is equivalent to
membership in $\Lie(U^\circ)$.
The projection $H/B^\circ\to H/P$ identifies the copies indexed by
$M/A$; the $\exp(\C\sigma)$-fixed cosets in $H/P$ are indexed by $W_H$.

Set
\[
 M(y)=\{(h,z)\in H\times\C^\times:
                         \Ad(h)y=z^2y\}.
\]
Thus $M(y)=Z_{H\times\C^\times}(y)$ for the action
$(h,z)\cdot x=z^{-2}\Ad(h)x$, in the notation of
\cite[equation~(3.1)]{Sol21}.
This group acts on $\cP_y$ by $(h,z)\cdot gP=hgP$.
We use complex coefficients throughout: $H^*_L$ denotes equivariant
cohomology and $H_*^L$ denotes equivariant Borel--Moore homology.
The relation $[\sigma,y]=2ry$ places the semisimple element
$(\sigma,r)$ in $\Lie(M(y)^\circ)$.  Let $\C_{\sigma,r}$ be the
one-dimensional $H^*_{M(y)^\circ}(\mathrm{pt})$-module given by evaluation
at its semisimple conjugacy class in the reductive quotient of this
Lie algebra; see \cite[equations~(3.3)--(3.4)]{Sol21}.
Define
\[
 E(y,\sigma,r)=
 \C_{\sigma,r}\otimes_{H^*_{M(y)^\circ}(\mathrm{pt})}
 H_*^{M(y)^\circ}(\cP_y,\C).
\]
The action of $Z_{M(y)}(\sigma,r)$ on the specialized homology
factors through its component group.  The natural isomorphism
\[
 \pi_0\bigl(Z_{M(y)}(\sigma,r)\bigr)
 \cong \pi_0\bigl(Z_H(\sigma,y)\bigr)
\]
therefore equips $E(y,\sigma,r)$ with an action of
$\pi_0(Z_H(\sigma,y))$; see
\cite[Section~3, equations~(3.5)--(3.7)]{Sol21}.

\begin{definition}[Eligible enhancements]
\label{def:eligible-enhancement}
We call $\rho\in\Irr\bigl(\pi_0(Z_H(\sigma,y))\bigr)$ \emph{eligible}
for the chosen quasi-Levi datum $(M,0,\triv)$ if
\[
 \Hom_{\pi_0(Z_H(\sigma,y))}
 \bigl(\rho,E(y,\sigma,r)\bigr)\ne0.
\]
\end{definition}

Equivalently, $\rho$ is eligible precisely when the cuspidal
quasi-support attached to $(y,\sigma,r,\rho)$ is $H$-conjugate
to $(M,0,\triv)$; see \cite[Theorem~1.6]{AMS17}.

The weights at the fixed points are as follows.\nopagebreak[4]

\begin{lemma} 
\label{lem:fixed-locus}

The $\exp(\C\sigma)$-fixed locus is the finite reduced set
\[
 \cP_y^\sigma=\{wP:w\in W_H,
                     \Ad(w^{-1})y\in\Lie(U^\circ)\}.
\]
The condition is independent of the representative of $w$ in $N_H(A)$,
since $M$ normalizes $U^\circ$.
There is an identification compatible with the polynomial and
component-group actions
\begin{equation}\label{eq:fixed-locus-homology}
 E(y,\sigma,r)\cong H_*(\cP_y^\sigma,\C).
\end{equation}
On the summand at $wP$, the polynomial character is
\begin{equation}\label{eq:fixed-point-character}
 p\longmapsto p(w^{-1}\sigma),
 \qquad p\in\operatorname{Sym}(\Lie(A)^*).
\end{equation}
\end{lemma}

\begin{proof}
Let $S$ be the Zariski closure of $\exp(\C\sigma)$ in $A$.
Since $\exp(\C\sigma)$ is connected, $S$ is a subtorus of $A$.
If $gP$ is $S$-fixed, then $S\subset gPg^{-1}$; since $S$ is
connected, $S\subset gB^\circ g^{-1}$.  Regularity implies that this
Borel contains $A$.  Its maximal tori $A$ and $gAg^{-1}$ are
conjugate within it, so $gP$ has a representative in $N_H(A)$.
Conversely every such representative yields an $S$-fixed coset.
Two representatives yield the same coset precisely when they differ
by $N_H(A)\cap P=M$.  Hence
\[
 (H/P)^S=N_H(A)/M=W_H.
\]
The fixed locus of a torus in the smooth variety $H/P$ is smooth;
here it is finite, hence reduced.  Imposing the Springer condition
gives the displayed subset, which is again reduced.  The same argument
applies to $\exp(\sigma)$.

For $P=MU^\circ$ as in \eqref{eq:quasi-parabolic-choice}, specialization
in the equivariant construction identifies the module with fixed-locus
homology by \cite[Proposition~3.5(a)]{Sol21}.  Restriction of an equivariant
polynomial class at $wP$ evaluates it on
$\Ad(w^{-1})\sigma=w^{-1}\sigma$.  This gives
\eqref{eq:fixed-point-character}, compatibly with the action of
$Z_H(\sigma,y)$ on the fixed points and on the coefficients.
\end{proof}

Return to $H=G_u^\vee$ and $M=M_u$.  Under this specialization,
the group $W_H$ defined above is precisely $W_u=N_H(A)/M_u$.
Lemma~\ref{lem:full-centralizer}
implies $M_u=A Z(H)$, since $Z(G^\vee)^{W_F}\subset Z(H)$.
We verify \eqref{eq:borel-normalization}.  By conjugacy of maximal
tori in $H^\circ$, every connected component of $H$ contains an
element $h\in N_H(A)$.  Since the Weyl group $N_{H^\circ}(A)/A$
acts transitively on the Borel subgroups of $H^\circ$ containing $A$,
there exists $n\in N_{H^\circ}(A)$ such that
\[
 (nh)B_u^{\vee,-}(nh)^{-1}=B_u^{\vee,-}.
\]
The element $nh$ lies in the same connected component as $h$ and
normalizes both $A$ and $B_u^{\vee,-}$.  It therefore normalizes
$M_u=Z_H(A)$ and $P=B_u^{\vee,-} Z(G^\vee)^{W_F}$, and preserves
the trivial equivariant local system.  Thus
\[
 H=N_H(P,M_u,\triv)H^\circ.
\]

For $K\subseteq\Sigma_v$ put
\begin{equation}\label{eq:NK}
 N_K=\sum_{\beta\in K}X_\beta,
 \qquad 0\ne X_\beta\in\mathfrak g_\beta,
 \qquad r=\frac{\log(q_F)}2.
\end{equation}
Then $[-\sigma_v,N_K]=2rN_K$.

\begin{lemma}[The component-group action]\label{lem:eligibility}
Each $N_K$ is nilpotent.  The trivial representation is the unique
eligible enhancement for the quasi-Levi datum $(M_u,0,\triv)$ for
each of the triples
\[
 (N_K,-\sigma_v,r)\qquad\text{and}\qquad(N_K,\sigma_v,-r).
\]
More precisely,
\begin{equation}\label{eq:pair-centralizer-with-centre}
 \begin{gathered}
 Z_H(\sigma_v,N_K)=Z_H(-\sigma_v,N_K)=Z(G^\vee)^{W_F} Z_A(N_K),
 \\
 Z_A(N_K)=\bigcap_{\beta\in K}\ker(\beta:A\to\C^\times),
 \end{gathered}
\end{equation}
and $\pi_0(Z_H(\sigma_v,N_K))$ acts trivially on both
\[
 E(N_K,-\sigma_v,r)\qquad\text{and}\qquad E(N_K,\sigma_v,-r).
\]
\end{lemma}

\begin{proof}
All roots in $\Sigma_v$ take the same negative real value on
$\sigma_v$.  They belong to a positive system, so $N_K$ lies in
the nilpotent radical of a Borel subgroup containing $A$.  It is
nilpotent, and the corresponding fixed Springer locus is nonempty.

By \eqref{eq:full-centralizer-corrected}, an element centralizing
$\sigma_v$ is of the form $za$, with $z\in Z(G^\vee)^{W_F}$ and
$a\in A$.
The element $z$ acts trivially on $\Lie(H)$, and $a$ fixes $N_K$
precisely when every root in its support takes value $1$ on $a$.
This proves \eqref{eq:pair-centralizer-with-centre}.

Every fixed coset is represented by $n\in N_H(A)$.  For
$c\in Z(G^\vee)^{W_F} Z_A(N_K)$ we have
$n^{-1}cn\in Z(G^\vee)^{W_F} A=M_u\subset P$, hence $c\cdot nP=nP$.
The coefficient system has trivial equivariant fibre action.
Thus $Z_H(\sigma_v,N_K)$ acts trivially on the finite fixed
locus and its homology.  Lemma~\ref{lem:fixed-locus} gives the
same conclusion for $E(N_K,-\sigma_v,r)$.
This nonzero module has only the trivial component-group isotypic
part, proving uniqueness of the eligible enhancement.  The same
argument applies to $(N_K,\sigma_v,-r)$: changing both signs
leaves the fixed loci and the component-group actions unchanged.
\end{proof}

Equation~\eqref{eq:pair-centralizer-with-centre} also shows that
the Springer condition for $H^\circ$, together with trivial action
of $Z(G^\vee)^{W_F}$, forces the enhancement to be trivial.
This does not imply that the parameter component group is trivial.

Consequently
\[
 \Hom_{\pi_0(Z_H(-\sigma_v,N_K))}
       (\triv,E(N_K,-\sigma_v,r))
 =E(N_K,-\sigma_v,r).
\]

\begin{remark}[Enhancement conventions]
For the component group in \eqref{eq:geometric-triple}, representations
of a quotient are compared with geometric enhancements by inflation
along the quotient map.
A representation of the quotient is trivial if and only if its
inflation is trivial.  More generally, an irreducible representation
of the full component group descends to the quotient precisely when
the kernel acts trivially.

The comparison with the Langlands component group $\mathcal R_\phi$
defined in the introduction
will be made explicit in Section~\ref{sec:parameters}, where we will
deduce the triviality of the Langlands enhancement from the triviality
of the geometric enhancement proved above.
\end{remark}

Write $t_w$ for the standard basis element indexed by $w\in W_H$
in the group algebra $\C[W_H]$ inside the graded Hecke algebra.
We identify $\xi\in\Lie(A)^*$ with its degree-one polynomial generator.
We use the Iwahori--Matsumoto automorphism defined by
\[
 \IM(t_w)=\det(w)t_w\quad(w\in W_H),
 \qquad \IM(\xi)=-\xi\quad(\xi\in\Lie(A)^*).
\]
Here $\det(w)$ is the determinant of the action of $w$ on $\Lie(A)$
for every $w\in W_H=W_u$, including $w\in\Gamma_{u>0}$;
see \cite[equation~(6-16)]{Sol23}.
The involution preserves the parameter $r=\log(q_F)/2$
fixed in \eqref{eq:NK}.  Set
\[
 E_K=\IM^*E(N_K,-\sigma_v,r).
\]

\begin{proposition}[The $\sigma_v$-weight of a standard module]
\label{prop:standard-weight}
For $K\subseteq\Sigma_v$,
\begin{equation}\label{eq:standard-weight-indicator}
 \dim(E_K)_{\sigma_v}=
 \begin{cases}
 1,&K\subseteq\Sigma_v^+,\\
 0,&K\not\subseteq\Sigma_v^+.
 \end{cases}
\end{equation}
\end{proposition}

\begin{proof}
Apply Lemma~\ref{lem:fixed-locus} to $(N_K,-\sigma_v,r)$.
After pullback by $\IM$, the fixed coset indexed by $w\in W_H$
has polynomial character $p\mapsto p(w^{-1}\sigma_v)$.
The stabilizer of $\sigma_v$ in $W_H$ is
trivial by Lemma~\ref{lem:full-centralizer}, so only $w=1$
contributes the weight $\sigma_v$.  The coset $P$ satisfies the
Springer condition precisely when $N_K\in\Lie(B_u^{\vee,-})$,
or equivalently $K\subseteq\Sigma_v^+$ by
\eqref{eq:P-definition}.  Its homology summand is one-dimensional.
\end{proof}

\section{Irreducible graded Hecke modules}
\label{sec:constituents}

The monodromy eigenspace in \eqref{eq:full-monodromy-space} has
linearly independent root characters.  Lemma~\ref{lem:eligibility}
determines the enhancement.  It remains to find the
support of the nilpotent element.  We combine the standard-module
weight calculation with the geometric multiplicity formula to
determine that support.

\begin{lemma}
\label{lem:coordinate-orbits}
Let $\Psi$ be a finite linearly independent set of characters of a complex
algebraic torus $T_0$.  For each $\beta\in\Psi$, let $V_\beta$ be
a one-dimensional $T_0$-module of weight $\beta$.
For fixed $0\ne X_\beta\in V_\beta$ and $J\subseteq\Psi$, set
$N_J=\sum_{\beta\in J}X_\beta$ in
$\bigoplus_{\beta\in\Psi}V_\beta$.  Then the $T_0$-orbit of $N_J$
consists of the vectors with support exactly $J$, and
\[
 \overline{T_0N_J}=\bigoplus_{\beta\in J}V_\beta.
\]
For $J,K\subseteq\Psi$, one has
$\overline{T_0N_K}\subseteq\overline{T_0N_J}$ if and only if
$K\subseteq J$, and every orbit closure is a smooth affine
coordinate space.
\end{lemma}

\begin{proof}
The homomorphism
$T_0\longrightarrow(\C^\times)^\Psi$,
$t\longmapsto(\beta(t))_{\beta\in\Psi}$, is surjective: its pullback
on character lattices sends the standard basis to the linearly
independent characters in $\Psi$, hence is injective.  The nonzero
coordinates of $N_J$ can therefore be scaled independently.
Taking limits at zero gives the displayed closure and the containment
criterion.  Each orbit closure is therefore a coordinate subspace
isomorphic to $\mathbb A_{\C}^{|J|}$, hence smooth.
This argument does not require the character sublattice generated
by $\Psi$ to be primitive.
\end{proof}

For $J\subseteq\Sigma_v$, let $L_J$ be the unique irreducible
quotient of $E_J$.  It exists and is unique by
\cite[Theorem~3.2(a)]{Sol21}, since $r>0$ and pullback by $\IM$
is an exact equivalence.
With the determinant normalization of $\IM$ specified above,
$L_J$ is the irreducible module attached to
$(\sigma_v,N_J,\triv)$, where
$[\sigma_v,N_J]=-\log(q_F)N_J$; see
\cite[Theorem~6.9 and equation~(6-16)]{Sol23}.

The eigenspace \eqref{eq:full-monodromy-space} consists of nilpotent
elements, since it lies in the nilpotent radical used in the proof of
Lemma~\ref{lem:eligibility}.  The group
$Z_H(-\sigma_v)=M_u=A Z(G^\vee)^{W_F}$ acts on this eigenspace with
the same orbits as $A$, because $Z(G^\vee)^{W_F}$ acts trivially by
conjugation.
Lemma~\ref{lem:coordinate-orbits} identifies these orbits with the
subsets $J\subseteq\Sigma_v$.  For each $J$, the trivial
isotypic component of $E(N_J,-\sigma_v,r)$ is nonzero and all other
isotypic components vanish by Lemma~\ref{lem:eligibility}.  Hence, for every
$J\subseteq\Sigma_v$, the cuspidal quasi-support attached to
$(N_J,-\sigma_v,r,\triv)$ is $H$-conjugate to $(M_u,0,\triv)$;
see \cite[Propositions~3.5(a) and~4.1]{Sol21}.
The geometric multiplicity formula in
\cite[Proposition~5.1]{Sol21} therefore applies to every pair
of these orbits, including when $H$ is disconnected.

Write $\mathcal O_J=A N_J$, and let
$j_J:\mathcal O_J\hookrightarrow\overline{\mathcal O_J}$
be the inclusion.  The trivial enhancement corresponds to
the constant local system $\underline{\C}_{\mathcal O_J}$
with its natural $M_u$-equivariant structure; see
\cite[Section~4, equation~(4.1)]{Sol21}.
Define
\[
 \operatorname{IC}_J
 :=(j_J)_{!*}\bigl(\underline{\C}_{\mathcal O_J}[|J|]\bigr),
\]
where $(j_J)_{!*}$ denotes perverse intermediate extension.
We regard this complex as extended by zero to the eigenspace
\eqref{eq:full-monodromy-space}.
By Lemma~\ref{lem:coordinate-orbits}, its support
$\overline{\mathcal O_J}$ is smooth of complex dimension $|J|$.
Hence
\[
 \operatorname{IC}_J
 \cong \underline{\C}_{\overline{\mathcal O_J}}[|J|].
\]

Let $i_K$ denote the inclusion of the point $N_K$ into this
eigenspace.  The functors $i_K^*$ and $i_K^!$ give the stalk
and costalk, respectively.  For $K\subseteq J$, we have
\[
 i_K^*\operatorname{IC}_J\cong\C[|J|],
 \qquad
 i_K^!\operatorname{IC}_J\cong\C[-|J|],
\]
and both complexes vanish when $K\not\subseteq J$.
The group $\pi_0(Z_H(-\sigma_v,N_K))$ acts trivially on
their cohomology: this follows from the natural equivariant
structure for the stalk, and additionally from preservation
of the complex orientation for the costalk.

The geometric multiplicity formula gives
\[
 [E_K:L_J]
 =
 \sum_{m\in\mathbb Z}
 \dim\Hom_{\pi_0(Z_H(-\sigma_v,N_K))}
 \bigl(\triv,H^m(i_K^*\operatorname{IC}_J)\bigr)
 =
 \begin{cases}
  1,& K\subseteq J,\\
  0,& K\not\subseteq J;
 \end{cases}
\]
see \cite[Section~5, equation~(5.2) and
Proposition~5.1(a)]{Sol21}.
Here $[E_K:L_J]$ denotes the Jordan--H\"older multiplicity,
and pullback by $\IM$ preserves these multiplicities.
Thus
\begin{equation}\label{eq:standard-multiplicities}
 [E_K]=\sum_{J\supseteq K}[L_J]
\end{equation}
in the Grothendieck group of finite-dimensional modules over
the graded Hecke algebra.

\begin{proposition}[The $\sigma_v$-weight of an irreducible module]\label{prop:irreducible-weight}
Among the modules $L_J$, exactly $L_{\Sigma_v^+}$ contains
the generalized weight $\sigma_v$, and that weight has multiplicity
one.
\end{proposition}

\begin{proof}
Taking the generalized $\sigma_v$-weight space is exact on
finite-dimensional modules: the polynomial action on a short exact
sequence factors through a common finite-dimensional commutative
Artinian quotient, whose local idempotent at $\sigma_v$ projects
each term onto its generalized weight space.
Put $m_J=\dim(L_J)_{\sigma_v}$.  Taking generalized weight-space
dimensions in
\eqref{eq:standard-multiplicities} and using Proposition~\ref{prop:standard-weight}
gives
\begin{equation}\label{eq:zeta-system}
 \mathbf 1_{K\subseteq\Sigma_v^+}
 =\sum_{J\supseteq K}m_J.
\end{equation}
M\"obius inversion on the subset lattice yields
\begin{equation}\label{eq:mobius}
 m_J=\sum_{J\subseteq K\subseteq\Sigma_v}
       (-1)^{|K|-|J|}\mathbf 1_{K\subseteq\Sigma_v^+}.
\end{equation}
If $J\not\subseteq\Sigma_v^+$, every summand vanishes.
If $J\subseteq\Sigma_v^+$, the nonzero summands are indexed
by the subsets of $\Sigma_v^+\setminus J$, and their sum is
$1$ when this set is empty and $0$ otherwise.
Thus $m_J=1$ for $J=\Sigma_v^+$ and $m_J=0$ otherwise.
\end{proof}

\begin{proposition}[The parameter of $V_{\mathrm{gr}}$]
\label{prop:selected-triple}
The reduced module $V_{\mathrm{gr}}$ is isomorphic to
$L_{\Sigma_v^+}$.  With semisimple entry fixed equal to
$\sigma_v$, its geometric triple is conjugate to
\begin{equation}\label{eq:selected-triple}
 \left(
  \sigma_v,
  \sum_{\substack{a\in\cS(\chi)\\
                   \Gamma\subset\mathcal H_a^+}}
       X_{-v^{-1}\iota(a)},
  \triv
 \right).
\end{equation}
\end{proposition}

\begin{proof}
The nilpotent entry $N$ of the triple of $V_{\mathrm{gr}}$ lies in
the full eigenspace \eqref{eq:full-monodromy-space}.
By Lemma~\ref{lem:coordinate-orbits}, it is $A$-conjugate to
$N_J$ for a unique subset $J\subseteq\Sigma_v$.  This conjugation fixes
$\sigma_v$.  Lemma~\ref{lem:eligibility} forces the enhancement to
be trivial, so $V_{\mathrm{gr}}\cong L_J$.

The selected weight of $V_{\mathrm{gr}}$ has multiplicity one by
\eqref{eq:selected-sigma-weight}.  Proposition~\ref{prop:irreducible-weight}
therefore forces $J=\Sigma_v^+$.
Finally $\Sigma_v^+$ is described by
\eqref{eq:P-definition}, and $\beta_{a,v}=-v^{-1}\iota(a)$,
giving \eqref{eq:selected-triple}.
\end{proof}

\section{Proof of the main theorem and consequences}
\label{sec:parameters}

We now complete the proof of Theorem~\ref{thm:main}.  Recall that,
for every $\Gamma\in\Ch(\chi)$, the asserted enhanced parameter of
$\pi_\Gamma$ has a representative satisfying
\[
 \widetilde\phi_\Gamma=\widehat\chi,
 \qquad
 N_\Gamma\sim
 \sum_{\substack{a\in\cS(\chi)\\
                  \Gamma\subset\mathcal H_a^+}}
       X_{-\iota(a)},
 \qquad
 \rho_\Gamma=\triv.
\]
Here conjugacy is taken under
$Z_{G^\vee}\bigl(\widehat\chi(W_F)\bigr)^\circ$, and each displayed
root vector is nonzero.  We first explain how the geometric parameters
obtained in the preceding sections translate into enhanced Langlands
parameters, including the comparison of enhancements.

Under Solleveld's Whittaker-normalized correspondence, the enhanced
parameter $(\phi,\rho)$ associated with a geometric triple with
semisimple entry $\sigma_v$ has cuspidal support represented by the
torus parameter $t_v=u\exp(\sigma_v)$; see
\cite[proof of Lemma~6.10]{Sol23}.
For principal-series parameters, the torus cuspidal-support parameter and
$\widetilde\phi$ agree up to conjugacy; see
\cite[Section~4, equation~(4-1)]{Sol23}.  With the representatives fixed
here, that parameter is $t_v=\widehat{\chi^v}$.
The parameter translation identifies the nilpotent element with
$\log\phi\left(1,\begin{psmallmatrix}1&1\\0&1\end{psmallmatrix}\right)$.
To compare enhancements, put
\[
 C_\phi=\pi_0 Z_{G^\vee}\bigl(\phi(W_F\times\{1\}),N_\phi\bigr).
\]
The inclusion of the centralizer of $\phi(\mathrm{SL}_2(\C))$
into the centralizer of $N_\phi$ induces an isomorphism on
component groups.  Consequently there is a natural surjection
$C_\phi\twoheadrightarrow\mathcal R_\phi$, and an LLC enhancement
is compared with a geometric enhancement by inflation along this
surjection; see \cite[Section~4]{Sol23}.
To make the comparison with $\pi_0(Z_H(\sigma_v,N_\phi))$ explicit, put
\[
 h_\phi=d(\phi|_{\mathrm{SL}_2})
       \begin{pmatrix}1&0\\0&-1\end{pmatrix},
 \qquad \sigma_0=\sigma_v+\frac{\log(q_F)}2h_\phi.
\]
For the compatible representative supplied by the parameter construction,
$h_\phi\in\Lie(H)$ and $[\sigma_v,h_\phi]=0$.  Moreover,
\[
 [\sigma_v,N_\phi]=-\log(q_F)N_\phi,
 \qquad [h_\phi,N_\phi]=2N_\phi.
\]
Thus $\sigma_0$ is semisimple and $[\sigma_0,N_\phi]=0$.
The logarithmic parameter construction and the change from
$\sigma_0$ to $\sigma_v$ give, respectively, the isomorphisms
\[
 C_\phi\cong\pi_0\bigl(Z_H(\sigma_0,N_\phi)\bigr)
 \cong \pi_0\bigl(Z_H(\sigma_v,N_\phi)\bigr);
\]
see \cite[proof of Theorem~3.8(a)]{AMS17} and
\cite[equations~(3.6)--(3.7)]{Sol21}.
By Lemma~\ref{lem:eligibility}, the corresponding geometric enhancement
is trivial.  Since $C_\phi\twoheadrightarrow\mathcal R_\phi$ is surjective,
the LLC enhancement itself is trivial.
These comparisons are part of the parameter correspondence of
\cite[Theorem~6.9, Lemma~6.10, and Theorem~7.1]{Sol23}.

\begin{proof}[Proof of Theorem~\ref{thm:main}]
Fix $\Gamma\in\Ch(\chi)$ and choose $v$ by
\eqref{eq:selected-v}.  Sections~\ref{sec:hecke}--\ref{sec:constituents} transport
$\pi_\Gamma$ through the Morita equivalence of Section~\ref{sec:hecke},
the comparison \eqref{eq:affine-comparison}, and the inverse reductions
\eqref{eq:reductions}.
Proposition~\ref{prop:selected-triple} identifies the resulting geometric
triple.  With infinitesimal representative $t_v$, the enhanced parameter therefore
has
\begin{equation}\label{eq:v-chart}
 \widetilde\phi_\Gamma=\widehat{\chi^v}=t_v,
 \qquad
 N_\Gamma\sim
 \sum_{\substack{a\in\cS(\chi)\\
                  \Gamma\subset\mathcal H_a^+}}
 X_{-v^{-1}\iota(a)},
 \qquad
 \rho_\Gamma=\triv.
\end{equation}

By Weyl equivariance of the torus LLC, conjugating the enhanced
parameter by a representative of $v$ sends $\widehat{\chi^v}$ to
$\widehat\chi$, sends each monodromy root $-v^{-1}\iota(a)$ to
$-\iota(a)$, and preserves the trivial enhancement.  This gives
\eqref{eq:main-formula}.

Finally, Lemma~\ref{lem:coordinate-orbits}, applied to the independent
characters $-\iota(a)$ of the centralizer torus, shows that all choices
of nonzero vectors in the displayed root spaces are conjugate.
\end{proof}

\begin{proof}[Proof of Corollary~\ref{cor:socle}]
Proposition~\ref{prop:socle-chamber} gives
$\soc I(\chi)=\pi_{\Gamma_{\mathrm{soc}}}$.  For a fixed
$a\in\cS(\chi)$, the connected chamber $\Gamma_{\mathrm{soc}}$ is disjoint
from $\mathcal H_a$, so $x\mapsto\langle x,a\rangle$ has constant sign on it.
Since $C^+\subset\Gamma_{\mathrm{soc}}$,
\[
 \Gamma_{\mathrm{soc}}\subset\mathcal H_a^+
 \quad\Longleftrightarrow\quad
 C^+\subset\mathcal H_a^+.
\]
Apply Theorem~\ref{thm:main} to $\Gamma_{\mathrm{soc}}$ and substitute this
equivalence into its support formula.  The result is
\eqref{eq:socle-formula}, including the conjugacy and enhancement assertions.
\end{proof}

\begin{corollary}[Vanishing of socle monodromy]
\label{cor:zero-monodromy}
The socle has zero monodromy if and only if
\begin{equation}\label{eq:zero-monodromy-criterion}
 \langle x,a\rangle<0
 \qquad\text{for every }x\in C^+\text{ and }a\in\cS(\chi).
\end{equation}
In this case its parameter is $\phi_{\mathrm{soc}}(w,g)=\widehat\chi(w)$,
with trivial enhancement.
\end{corollary}

\begin{proof}
The root spaces in \eqref{eq:socle-formula} are distinct, so the sum
vanishes exactly when its indexing set is empty.  Each reducibility
coroot has constant nonzero sign on $C^+$, which gives
\eqref{eq:zero-monodromy-criterion}.  If $N_{\mathrm{soc}}=0$, the
algebraic homomorphism on $\mathrm{SL}_2$ is trivial.  The last assertion
then follows from $\widetilde\phi_{\mathrm{soc}}=\widehat\chi$ and
$\rho_{\mathrm{soc}}=\triv$.
\end{proof}

We now compare \eqref{eq:zero-monodromy-criterion} with the hypotheses
of \cite[Definitions~6.6.4--6.6.5]{BC09}.

\begin{proposition}[The conditions of Bella\"iche--Chenevier]
\label{prop:bc-comparison}
Let $F'/F$ be a finite Galois extension that splits $G$, and put
\[
 \chi'=\chi\circ\mathrm{Nm}_{T(F')/T(F)}.
\]
Here $\mathrm{Nm}_{T(F')/T(F)}(t)=
\prod_{\gamma\in\operatorname{Gal}(F'/F)}\gamma(t)$ for $t\in T(F')$.
Positivity for the absolute coroots is taken with respect to $B_{F'}$.
Suppose that $\chi'$ is regular for the absolute Weyl group and that
\begin{equation}\label{eq:bc-condition}
 \chi'\circ\alpha^\vee\ne|\,\cdot\,|_{F'}
 \qquad\text{for every positive absolute coroot }\alpha^\vee
 \text{ of }G_{F'}.
\end{equation}
Then $\chi$ is regular and $N_{\mathrm{soc}}=0$.
\end{proposition}

\begin{proof}
The relative Weyl group embeds in the absolute Weyl group, and its action
commutes with the norm map.  Thus any relative Weyl element fixing $\chi$
also fixes $\chi'$, and must be the identity.  In particular, the main
theorem applies to $\chi$.

Choose the representative in \eqref{eq:socle-formula}, with
$\widetilde\phi_{\mathrm{soc}}=\widehat\chi$.  Every coroot occurring
in the sum is positive relative to $B$.  Since
$B^{\vee,-}$ is opposite to $B^\vee$, the root
convention in the theorem places the corresponding monodromy vectors in
$\Lie(U^\vee)$.

Suppose that $N_{\mathrm{soc}}\ne0$, and choose a nonzero component in
an absolute positive root space $\Lie(G^\vee)_\gamma$.
Write $\widehat\chi(w)=(t_\chi(w),w)$.  On $W_{F'}$ the pinned action
is trivial, so $t_\chi|_{W_{F'}}$ is a torus homomorphism.  Comparing
$\gamma$-components in the Weil--Deligne relation gives
\[
 \gamma(t_\chi(w))=\|w\|_{F'}\qquad(w\in W_{F'}).
\]
The torus local Langlands correspondence identifies restriction to
$W_{F'}$ with pullback by the norm.  If $a_\gamma$ is the positive
absolute coroot of $G_{F'}$ corresponding to the dual root $\gamma$,
the displayed equality therefore says
\[
 \chi'\circ a_\gamma=|\,\cdot\,|_{F'}.
\]
Here $\|\,\cdot\,\|_F|_{W_{F'}}=\|\,\cdot\,\|_{F'}$, so no
rescaling of the norm is involved.  This contradicts
\eqref{eq:bc-condition}, and proves the proposition.
\end{proof}

The argument uses norm compatibility for torus parameters.  The
compatibility between base change and the socles over $F$ and $F'$
discussed in \cite[Remark~6.6.6(3)]{BC09} is a separate statement.

\begin{corollary} [Distinctness of the parameters]
\label{cor:packet-separation}
The parameters $\phi_\Gamma$, $\Gamma\in\Ch(\chi)$, are pairwise
nonconjugate under $G^\vee$.  Consequently, a single unenhanced $L$-parameter labels
at most one constituent of the fixed regular principal series $I(\chi)$.
\end{corollary}

\begin{proof}
If $\phi_\Gamma$ and $\phi_{\Gamma'}$ are conjugate, the conjugation
identifies their trivial enhancements.  Theorem~\ref{thm:main} therefore
gives the same enhanced parameter for both constituents.  Injectivity of
the principal-series correspondence of \cite[Theorem~7.1]{Sol23} gives
$\pi_\Gamma\cong\pi_{\Gamma'}$.  The chamber parametrization in
Theorem~\ref{thm:chambers} then gives $\Gamma=\Gamma'$.
\end{proof}

The corollary concerns constituents of the fixed regular principal
series $I(\chi)$.  It makes no assertion about other members of the
corresponding $L$-packets.

\begin{remark}
The rank-one and monodromy signs are fixed in Section~\ref{sec:rank-one}.
The Iwahori--Matsumoto twist is included in $E_K$ and contributes
no further sign to \eqref{eq:main-formula}.
\end{remark}

\subsection{The case of \texorpdfstring{$\mathrm{GL}_2$}{GL2}}
In the matrix examples below, $e_i$ denotes the $i$th standard basis
vector and $E_{ij}$ the matrix unit with its only nonzero entry, equal
to $1$, in position $(i,j)$.
Let $G=\mathrm{GL}_2$, let $B$ and $B^\vee$ be upper
triangular, so that $B^{\vee,-}$ is lower triangular, and write
$\nu=|\,\cdot\,|$.  Let $\alpha$ be the
usual positive root of $G^\vee$ and take
$\chi_+=(\nu^{1/2},\nu^{-1/2})$.  The oriented coroot is the
positive coroot $a$, and
\[
 h_a=\operatorname{diag}(\varpi_F^{-1},\varpi_F),\qquad
 d_a=1,\qquad \lambda_a=\lambda_a^*=1,\qquad
 \iota(a)=-\alpha.
\]
The chamber $C^+$ is on the positive side of $a$, so the socle
formula gives $N=E_{12}$.  With the chosen reciprocity,
\[
 \widehat\chi_+(\mathrm{Fr}_F)
       =\operatorname{diag}(q_F^{-1/2},q_F^{1/2}),\qquad
 \Ad\bigl(\widehat\chi_+(\mathrm{Fr}_F)\bigr)E_{12}
       =q_F^{-1}E_{12}.
\]
This is the Steinberg socle of $I(\chi_+)$; its other constituent
is the trivial quotient, with $N=0$.

For this socle, the choice of weight in the opposite Jacquet module
gives $v=w_0$.
There $\sigma_v=\operatorname{diag}(\log(q_F)/2,-\log(q_F)/2)$,
$\beta_{a,v}=-\alpha$, and the positive root for the lower
triangular $B_u^{\vee,-}$ is represented by $E_{21}$.
Conjugating the whole pair by $w_0$ returns the preceding
representative with monodromy $E_{12}$.
For the reversed character
$\chi_-=(\nu^{-1/2},\nu^{1/2})$, the oriented coroot is $-a$.
Its socle chamber $C^+$ is on the negative side, and the formula
gives the trivial socle with $N=0$.  The Steinberg quotient is
represented by
$(\operatorname{diag}(q_F^{1/2},q_F^{-1/2}),E_{21})$, which is
conjugate to the Steinberg pair above.  Reversing the inducing
character therefore exchanges the subrepresentation and quotient
while preserving their Langlands parameters.
More generally, a common unramified twist $\eta$ multiplies both
Frobenius matrices by $\eta(\varpi_F)$ and leaves the displayed
monodromy matrices unchanged.  Under conjugation back to the original representative, this scalar
remains $\eta(\varpi_F)$, as prescribed by
\eqref{eq:torus-reciprocity}.

\section{An example for unramified \texorpdfstring{$U_3$}{U3}}
\label{sec:unitary-example}

We compute the constituent parameters at both nonunitary reducibility
points of the unramified principal series of $U_3$.  The two points
lead to different centralizers, and one gives a nontrivial
Langlands component group.  We use the representative with infinitesimal
parameter $\widehat\chi$ and the reciprocity convention of
\eqref{eq:torus-reciprocity}.

Let $E/F$ be the unramified quadratic extension, write $q=q_F$, and use
the common uniformizer $\varpi=\varpi_F=\varpi_E$.  Then
$|\varpi|_E=q^{-2}$.  Write $a\mapsto\bar a$ for the nontrivial
automorphism of $E/F$.  Let $G=U_3(E/F)$ be quasi-split, with upper
triangular $B$, and write
\[
 T(F)=\{\operatorname{diag}(a,b,\bar a^{-1}):
                a,b\in E^\times,\ b\bar b=1\}.
\]
Let $\epsilon_E$ be the unramified quadratic character of $E^\times$,
with $\epsilon_E(\varpi)=-1$, and put
\[
 \xi_+=|\,\cdot\,|_E,
 \qquad \xi_- =\epsilon_E|\,\cdot\,|_E^{1/2}.
\]
Define $\chi_\pm(\operatorname{diag}(a,b,\bar a^{-1}))=\xi_\pm(a)$;
in particular, these characters are trivial on the middle factor.
Both are regular: the nonidentity relative Weyl element sends an
unramified character $\xi$ to $\xi^{-1}$, and
$\xi_\pm(\varpi)\ne\xi_\pm(\varpi)^{-1}$.

\subsection{The two reducibility points}

Set
\[
 h=\operatorname{diag}(\varpi^{-1},1,\varpi),
 \qquad z(\chi)=\chi(h).
\]
This is the primitive coordinate on the unramified character torus.
The Iwahori-spherical labels in this coordinate are
\[
 (\lambda,\lambda^*)=(3,1).
\]
These labels are given in \cite[Section~1, after (1-15)]{Sol23}.
Consequently the two nonunitary
reducibility values with $|z|>1$ are
\[
 z(\chi_+)=q^2,\qquad z(\chi_-)=-q.
\]
The maximal split torus has
$S(F)=\{\operatorname{diag}(t,1,t^{-1}):t\in F^\times\}$.
The indivisible positive relative root is determined by
\[
 \alpha(\operatorname{diag}(t,1,t^{-1}))=t.
\]
Its coroot is realized by
$H_a(t)=\operatorname{diag}(t^2,1,t^{-2})$.  Thus $d_a=2$ and
\[
 (e_a,b_a,s_a)=(1,2,4)\quad\text{for }\chi_+,
 \qquad (e_a,b_a,s_a)=(-1,1,2)\quad\text{for }\chi_-.
\]
Thus $s_a=d_ab_a$ in both cases.  The reducibility arrangement has
two chambers, $C^+$ and $C^-$, and $a$ is positive for the chosen
Borel.  The constituent indexed by $C^+$ is therefore the socle
and has nonzero monodromy; the constituent indexed by $C^-$ has
zero monodromy.

\subsection{The centralizers and their root spaces}

Take $G^\vee=\mathrm{GL}_3(\C)$, with diagonal torus $T^\vee$,
upper triangular $B^\vee$, and lower triangular $B^{\vee,-}$.
Let geometric Frobenius act through the pinned involution
\[
 \vartheta(g)=J\,{}^tg^{-1}J^{-1},\qquad
 J=\begin{pmatrix}0&0&1\\0&-1&0\\1&0&0\end{pmatrix}.
\]
The connected invariant torus is
\[
 A=\{\operatorname{diag}(t,1,t^{-1}):t\in\C^\times\},
 \qquad \varepsilon(\operatorname{diag}(t,1,t^{-1}))=t.
\]
The projection of $A$ to the Frobenius-coinvariant dual torus is the
isogeny $t\mapsto t^2$.  Indeed, a Frobenius representative
$\operatorname{diag}(t,1,t^{-1})\rtimes\vartheta$ squares to
$\operatorname{diag}(t^2,1,t^{-2})$, and reciprocity over $E$ sets
$t^2=\xi(\varpi)$.  The coordinate $z$ is therefore $t^{-2}$.
In particular its logarithmic differential is
\begin{equation}\label{eq:u3-coordinate-differential}
 \ell=-2\varepsilon.
\end{equation}
This is the isogeny used to define the logarithmic coordinates in
\cite[equation~(4-3)]{Sol23}.

Put $L=\log q$ and $D=\operatorname{diag}(i,1,-i)$.  The bounded parts
and the invariant logarithms may be taken as
\begin{align*}
 u_+(\mathrm{Fr}_F)&=1\rtimes\vartheta,
 &\sigma_+&=\operatorname{diag}(-L,0,L),\\*
 u_-(\mathrm{Fr}_F)&=D\rtimes\vartheta,
 &\sigma_-&=\operatorname{diag}(-L/2,0,L/2).
\end{align*}
Both bounded parameters are unramified.  The infinitesimal parameters are
\begin{equation}\label{eq:u3-infinitesimal}
 \widehat\chi_\pm(\mathrm{Fr}_F)
       =\exp(\sigma_\pm)u_\pm(\mathrm{Fr}_F).
\end{equation}
Squaring gives, respectively,
$\operatorname{diag}(q^{-2},1,q^2)$ and
$\operatorname{diag}(-q^{-1},1,-q)$, as required by
reciprocity over $E$.  The logarithms $\sigma_\pm$ lie in $\Lie(A)$;
they differ from the entrywise logarithms of representatives
$\operatorname{diag}(\xi(\varpi),1,1)\rtimes\vartheta$.

The centralizers $H_\pm=Z_{\mathrm{GL}_3}(u_\pm(W_F))$ of the
bounded parameters are
\begin{equation}\label{eq:u3-centralizers}
 H_+=\mathrm O(J),\quad H_+^\circ=\mathrm{SO}(J),
 \qquad
 H_-\cong\mathrm{SL}_2(\C)\times\{\pm1\}.
\end{equation}
For the last identification the $\mathrm{SL}_2$ factor acts on
$\langle e_1,e_3\rangle$ and $\{\pm1\}$ acts on $\langle e_2\rangle$.
Indeed, an element of $H_-$ commutes with
$(D\rtimes\vartheta)^2=\operatorname{diag}(-1,1,-1)$ and hence
preserves these two subspaces.  The form $DJ$ is alternating on the
first and symmetric on the second, giving precisely the indicated
symplectic and orthogonal factors.

The relevant positive root spaces, relative to the upper triangular
$B^\vee$, are consequently
\[
 \Lie(H_+^\circ)_{\varepsilon}
          =\C(E_{12}+E_{23}),\qquad
 \Lie(H_-^\circ)_{2\varepsilon}=\C E_{13}.
\]
Thus the positive roots of $H_+^\circ$ and $H_-^\circ$ are
$\varepsilon$ and $2\varepsilon$, respectively.  With the transported
lower triangular convention defined by $B^{\vee,-}$,
$\iota(a)=-\varepsilon$ for $\chi_+$ and
$\iota(a)=-2\varepsilon$ for $\chi_-$.  Thus
\eqref{eq:u3-coordinate-differential} gives exactly
$\ell=b_a\iota(a)$, with $b_a=2$ and $b_a=1$, respectively.
The full monodromy eigenspaces are
\begin{equation}\label{eq:u3-monodromy-spaces}
 \begin{aligned}
 \{X\in\Lie(H_+^\circ):[\sigma_+,X]=-LX\}
       &=\C(E_{12}+E_{23}),\\
 \{X\in\Lie(H_-^\circ):[\sigma_-,X]=-LX\}
       &=\C E_{13}.
 \end{aligned}
\end{equation}
The vector $E_{12}+E_{23}$ spans a single root space of
$H_+^\circ$ but involves two absolute root spaces of $\mathrm{GL}_3$.
Neither $E_{12}$ nor $E_{23}$ belongs to $\Lie(H_-)$.

\subsection{The parameters of the constituents}

Theorem~\ref{thm:main} gives the following representatives, with
infinitesimal parameters fixed by \eqref{eq:u3-infinitesimal}:
\[
\begin{array}{c|c|c|c}
 \text{inducing character}&\text{chamber}&N&\rho\\ \hline
 \chi_+&C^+\ (\text{socle})&E_{12}+E_{23}&\triv\\
 \chi_+&C^-&0&\triv\\
 \chi_-&C^+\ (\text{socle})&E_{13}&\triv\\
 \chi_-&C^-&0&\triv
\end{array}
\]
The commutators in \eqref{eq:u3-monodromy-spaces} give
$\Ad(\widehat\chi_\pm(\mathrm{Fr}_F))N=q^{-1}N$.
For the parameters with nonzero monodromy, let
$r_+:\mathrm{SL}_2\to\mathrm{SO}(J)$ be a principal homomorphism
with raising operator $E_{12}+E_{23}$ and diagonal element
$\operatorname{diag}(2,0,-2)$; the lowering operator is
$2(E_{21}+E_{32})$.  These three matrices form an $\mathfrak{sl}_2$
triple in $\Lie(\mathrm{SO}(J))$.  Let
$r_-:\mathrm{SL}_2(\C)\to H_-^\circ$ act in the usual way on
$\langle e_1,e_3\rangle$ and trivially on $e_2$, with raising operator
$E_{13}$ and diagonal element $\operatorname{diag}(1,0,-1)$.
The corresponding $L$-parameters are
\[
 \phi_\pm(w,x)=u_\pm(w)r_\pm(x).
\]
The images commute, and their infinitesimal parameters are exactly
\eqref{eq:u3-infinitesimal}.  Both $u_\pm$ are bounded and the
centralizers of the full parameters are finite.  Hence these socles
are discrete series by the temperedness and discreteness properties
of \cite[Lemmas~7.7--7.8]{Sol23}; the centre of $G(F)$ is compact.
The socle of $I(\chi_+)$ is the Steinberg representation, since
$\chi_+=\delta_B^{1/2}$, where $\delta_B$ is the modulus character
of $B(F)$:
\[
 \delta_B(t)=\left|\det\bigl(\Ad(t)|_{\Lie(U)}\bigr)\right|_F
 \qquad(t\in T(F)).
\]
The zero-monodromy constituents have the unbounded Weil parameters
$\widehat\chi_\pm$ and are not tempered.

The two parameters with nonzero monodromy have different component
groups.
For $N=E_{12}+E_{23}$, the full parameter centralizer is
$\{\pm I_3\}=Z(G^\vee)^{W_F}$, so $\mathcal R_{\phi_+}=1$.
For $N=E_{13}$ it is
\[
 Z_{G^\vee}(\phi_-(W_F\times\mathrm{SL}_2))
   =\{\operatorname{diag}(a,b,a):a,b\in\{\pm1\}\}.
\]
Dividing by $Z(G^\vee)^{W_F}=\{\pm I_3\}$ gives
\begin{equation}\label{eq:u3-nontrivial-component}
 \mathcal R_{\phi_-}\cong\Z/2\Z,\qquad \rho_{\phi_-}=\triv.
\end{equation}
For either zero-monodromy parameter, the centralizer is
$\{\operatorname{diag}(t,b,t^{-1}):t\in\C^\times,b=\pm1\}$;
its quotient by $\{\pm I_3\}$ is connected, so the component group
is trivial.  Thus triviality of the enhancement in the chamber formula
does not assert triviality of the Langlands component group.

\begin{remark}[The centralizer of the infinitesimal parameter]
The centralizer of the zero-monodromy parameter computed above is
also $Z_{H_\pm}(\sigma_\pm)$.
This group has identity component $A$, and each of its components
has a representative in $Z(G^\vee)^{W_F}$.  It fixes every
$\sigma_\pm$-fixed Borel of $H_\pm^\circ$, and its quotient by
$Z(G^\vee)^{W_F}$ is a connected torus.  Thus regularity of the
inducing character does not imply that the full centralizer is
connected.
\end{remark}

\subsection*{Acknowledgements}
This problem was suggested by Prof. Dongwen Liu during the summer school ``Algebra and Number Theory 2026'' held by Peking University and the Academy of Mathematics and Systems Science of the Chinese Academy of Sciences. We would like to express our sincere gratitude to Prof. Shou-Wu Zhang and Prof. Liang Xiao,
and we thank the lecturers of the summer school for their support. Without their efforts, the summer school could not have been such a success.

\end{document}